\documentclass[12pt]{amsart}
  \usepackage{color}
  \usepackage{latexsym} 
  \usepackage[all]{xy}
  \usepackage{amsfonts} 
  \usepackage{amsthm} 
  \usepackage{amsmath} 
  \usepackage{amssymb}
  \usepackage{pifont}  
  \usepackage{enumerate}
  \usepackage{hyperref}

  \def\<{{\langle}} 
  \def\>{{\rangle}}

  \def\eps{\varepsilon}

  \def\note#1{{}}

  \def\note#1{} 

  \def\ev{{\rm ev}} 
   
    \def\ve{\mathrm{ve}}
   \def \hve{\widehat{\ve}}

 \def\d{{\rm d}}

  \def\beq{\begin{equation}} 
  \def\eeq{\end{equation}}

  \def\id{\mathrm{id}}

  \def\ot{{\otimes}}

 \def\htau{\hat{\tau}}
 \def\hchi{\hat{\chi}}
 \def\hpi{\hat{\pi}}
 \def\hvarpi{\hat{\varpi}}
  \def\homega{\hat{\omega}}

    \def\hT{\widehat{T}}
    \def\bhT{{\bf \widehat{T}}}

    \def\bhP{{\bf \widehat{P}}}

  \def \hlambda{\widehat{\lambda}}
    \def\ev{\mathrm{ev}}
   \def \hev{\widehat{\ev}}

   \def\bUpsilon{\overline{\Upsilon}}
     \def\bGamma{\overline{\Gamma}}
  \def\hrho{\hat{\rho}}
   \def\hsigma{\hat{\sigma}}
  \def\hlambda{\hat{\lambda}}

  \newcounter{zlist}

  \newcounter{blist}

  \newcounter{rlist}

\def\stac#1{\raise-.2cm\hbox{$\stackrel{\displaystyle\otimes}{\scriptscriptstyle{#1}}$}}

\def\cten#1{\raise-.2cm\hbox{$\stackrel{\displaystyle\widehat{\otimes}}
{\scriptscriptstyle{#1}}$}}

  \def\Label#1{\label{#1}\ifmmode\llap{[#1] }\else 
  \marginpar{\smash{\hbox{\tiny [#1]}}}\fi} 
  \def\Label{\label}

  \newtheorem{proposition}{Proposition}[subsection]
  \newtheorem{lemma}[proposition]{Lemma} 
  \newtheorem{corollary}[proposition]{Corollary} 
  \newtheorem{theorem}[proposition]{Theorem} 

  \theoremstyle{definition} 
  \newtheorem{definition}[proposition]{Definition}

  \theoremstyle{remark} 
  \newtheorem{remark}[proposition]{Remark}

  \newcounter{c} 
  \renewcommand{\[}{\setcounter{c}{1}$$} 
  \newcommand{\etyk}[1]{\vspace{-7.4mm}$$\begin{equation}\Label{#1} 
  \addtocounter{c}{1}} 
  \renewcommand{\]}{\ifnum \value{c}=1 $$\else \end{equation}\fi} 
\newcommand{\Mor}{{\sf Mor}}
\newcommand{\Dual}{{\sf Dual}}
\newcommand{\Herd}{{\sf Herd}}
\newcommand{\bHerd}{\overline{\sf Herd}}
\newcommand{\pretor}{{\sf PreTor}}

\newcommand{\EM}{{\sf EM}}

\newcommand{\cov}{{\rm cov}}

\newcommand{\Adj}{{\sf Adj}}
\newcommand{\Span}{{\sf Span}}

\def\Ab{{\sf Ab}}%{\mathfrak{Ab}}%{\underline{\underline{\rm Ab}}}

\def\ot{\otimes}

\def\id{\textrm{{\small 1}\normalsize\!\!1}}

\def\TT{{\mathbb T}}

\def\ZZ{{\mathbb Z}}

\def\tt{{\mathfrak Z}}

\newcommand{\Cc}{\mathcal{C}}
\newcommand{\Dd}{\mathcal{D}}

\newcommand{\Mm}{\mathcal{M}}

\newcommand{\Tt}{\mathcal{Z}}

\newcommand{\Ww}{\mathcal{W}}
\newcommand{\Xx}{\mathcal{X}}
\newcommand{\Yy}{\mathcal{Y}}

\def\m{{\sf m}}
\def\biota{{\mathsf{u}}}

\def\A{{\bf A}}
\def\B{{\bf B}}
\def\C{{\bf C}}
\def\D{{\bf D}}
\def\E{{\bf E}}
\def\F{{\bf F}}

\def\H{{\bf H}}

\def\P{{\bf P}}
\def\Q{{\bf Q}}

\def\S{{\bf S}}
\def\T{{\bf T}}

\def\X{{\bf X}}
\def\Y{{\bf Y}}

\def\*C{{}^*\hspace*{-1pt}{\Cc}}

\def\text#1{{\rm {\rm #1}}}

\def\ol{\overline}
\def\ul{\underline}

\def\EM#1#2{#1\sp{#2}}
\def\EMC#1#2{#1\sb{#2}}
\def\EMM#1{(\Xx,\Yy)\sp{#1}}%{\Mm\sp{#1}}
\def\EMAb#1{(\Ab,\Ab)\sp{#1}}
\def\EMB#1{(\Xx,\Tt\sb{\C})\sp{#1}}
\def\bv{\bar{v}}
\def\bw{\bar{w}}
\def\tA{{A^{\ot}}}
\def\tB{{B^{\ot}}}
\def\fr#1{F\sb{#1}}
\def\frc#1{F\sp{#1}}
\def\feps{\bar{\eps}}

\begin{document} 

 \title[A Beck-type theorem for a Morita context]{The Eilenberg-Moore category and a Beck-type theorem for a Morita context} 
 \author{Tomasz Brzezi\'nski}
 \address{ Department of Mathematics, Swansea University, 
  Singleton Park, \newline\indent  Swansea SA2 8PP, U.K.} 
  \email{T.Brzezinski@swansea.ac.uk}   
\author{Adrian Vazquez Marquez}
 \address{ Department of Mathematics, Swansea University, 
  Singleton Park, \newline\indent  Swansea SA2 8PP, U.K.} 
\email{397586@swansea.ac.uk} 
 \author{Joost Vercruysse}
   \address{ Faculty of Engineering, Vrije Universiteit Brussel (VUB), \newline\indent B-1050 Brussels, Belgium}
\email{jvercruy@vub.ac.be}

    \date{September 2009} 
  \begin{abstract} 
The  Eilenberg-Moore constructions and a Beck-type theorem for pairs of monads are described. More specifically,  a notion of a {\em Morita context} comprising of two monads, two bialgebra functors and two connecting maps is introduced. It is shown that in many cases equivalences between categories of algebras are induced by such Morita contexts. The Eilenberg-Moore category of representations of a Morita context is constructed. This construction allows one to associate two pairs of adjoint functors with right adjoint functors having a common domain or a {\em double adjunction} to a Morita context. It is shown that, conversely, every Morita context arises from a double adjunction. The comparison functor between the  domain of  right adjoint functors in a double adjunction  and the Eilenberg-Moore category of the associated Morita context is defined. The sufficient and necessary conditions for this comparison functor to be an equivalence (or for the {\em moritability} of a pair of functors with a common domain) are derived. 
  \end{abstract} 
  \maketitle 
  \tableofcontents

\section{Introduction}
In the last two decades, Hopf-Galois theory went through a series of generalisations, ultimately leading to the theory of Galois comodules over corings, which elucidates its relation with well-known Beck's  monadicity theorem.
This evolution has revived the interest among Hopf algebraists in the theory of (co)monads. Several aspects of Hopf(-Galois) theory have been reformulated in the framework of (co)monads so that further clarification of the underlying categorical mechanisms of the theory has been achieved. This has led in particular to two different approaches to the definition of a categorical (functorial) notion of a {\em herd} or {\em pre-torsor}, appearing almost simultaneously in \cite{BrzVer:herd} and \cite{BohMen:pretor}. The motivation for the present paper was to study the connection between these two approaches in more detail.

The definition of a pre-trosor in \cite{BohMen:pretor} takes a pair of adjunctions (with coinciding codomain category for the left adjoints) as a starting point. The aim of this paper is to show that there is a close relationship between pairs of adjunctions and functorial Morita contexts similar to the correspondence between single adjunctions and monads. 
In this sense the results presented here can be interpreted as a `two-dimensional' version of the latter correspondence. 
A key feature of this work is that it links aspects of the theory that are of more algebraic nature (Morita contexts) with aspects that are of more categorical nature (Beck's theorem). More precisely, we prove a version of Beck's theorem on precise monadicity in this `two-dimensional' setting and provide a categorical (monadic) version of classical Morita theory.
  
  These are the main results and the organisation of the paper. In Section~\ref{adjMor}  we recall from \cite{BohMen:pretor} the definition of the category of double adjunctions  on categories $\Xx$ and $\Yy$, $\Adj (\Xx,\Yy)$,  and introduce the category  $\Mor (\Xx,\Yy)$ of (functorial) Morita contexts. 
  
  In Section~\ref{sec.Beck} we describe functors connecting categories of double adjunctions and Morita contexts. More precisely we compare categories $\Adj (\Xx,\Yy)$ and $\Mor (\Xx,\Yy)$. First we define a functor $\Upsilon: \Adj (\Xx,\Yy)\to \Mor (\Xx,\Yy)$. To construct a functor in the converse direction, to each Morita context $\TT$ we associate its {\em Eilenberg-Moore category} $\EMM\TT$. This is very reminiscent of the classical Eilenberg-Moore construction of algebras of a monad (recalled in Section~\ref{sec.eil-moo}), and, in a way, is based on doubling of the latter. Objects in $\EMM\TT$ are two algebras, one for each monad in $\TT$, together with two connecting morphisms. Once $\EMM\TT$ is defined, two adjunctions, one between $\EMM\TT$ and $\Xx$ the other between $\EMM\TT$ and $\Yy$,  are constructed. This construction yields a functor $\Gamma: \Mor (\Xx,\Yy)\to \Adj (\Xx,\Yy)$. Next it is shown that the functors $(\Gamma, \Upsilon)$  form an adjoint pair, and that $\Gamma$ is a full and faithful functor. The counit of this adjunction is given by a {\em comparison functor} $K$ which compares the common category $\Tt$ in a double adjunction $\tt$ with the Eilenberg-Moore category of the associated Morita context $\TT = \Upsilon(\tt)$. A necessary and sufficent condition for the comparison functor to be an equivalence are derived. This is  closely related to the existence of colimits of diagrams of certain type in $\Tt$ and is a  Morita--double adjunction version of the classical Beck theorem (on precise monadicity). 
  
 In Section~\ref{sec.Moritath} we analyse which objects of $\Mor (\Xx,\Yy)$ describe equivalences between categories of algebras of monads. It is also proven that large classes of equivalences between categories of algebras are induced by  Morita contexts. 
  
  In Section~\ref{sec.examples} examples and special cases of the theory developed in preceding sections are given. In particular, it is shown how the main results of Section~\ref{sec.Beck} can be applied to a single adjunction leading to a new point of view on some aspects of descent theory.  The Eilenberg-Moore category associated to the module-theoretic Morita context is identified as the category of modules of the associated matrix Morita ring. This procedure can be imiated in the general case if we assume the existence and preservation of binary coproducts in all categories and by all functors involved. We also show that the theory developed in Sections~\ref{adjMor}--\ref{sec.Moritath} is applicable to pre-torsors and herd functors, thus bringing forth means for comparing pre-torsors with balanced herds. The paper is completed with comments on dual versions of constructions presented and with an outlook.
  
  Throughout the paper, the composition of functors is denoted by juxtaposition, the symbol $\circ$ is reserved for composition of natural transformations and morphisms. The action of a functor on an object or morphism is usually denoted by juxtaposition of corresponding symbols (no brackets are typically used). Similarly the morphism corresponding to a natural transformation, say  $\alpha$, evaluated at an object, say $X$, is denoted by a juxtaposition, i.e.\ by $\alpha X$. For an object $X$ in a category, we will use the symbol $X$ as well to denote the identity morphism on $X$.
Typically, but not exclusively, objects and functors are denoted by capital Latin letters, morphisms by small Latin letters and natural transformations by Greek letters. 

\section{Double adjunctions and Morita contexts}\label{adjMor}
The aim of this section is to recall the standard correspondence between adjoint functors and (co)monads and to introduce the main categories studied in the paper. 

\subsection{Adjunctions and (co)monads}\label{sec.eil-moo}
It is well-known from \cite{EM} that there is a close relationship between pairs of \emph{adjoint functors} $(L:\Xx\to \Yy, R:\Yy\to \Xx)$, \emph{monads} $\A$ on $\Xx$ and \emph{comonads} $\C$ on $\Yy$. Starting from an adjunction $(L,R)$ with unit $\eta:\Xx\to RL$ and counit $\varepsilon:LR\to \Yy$, the corresponding monad and comonad are $\A=(RL, R\varepsilon L, \eta)$ and $\C=(LR,L\eta R,\varepsilon)$. Starting from a monad $\A=(A,\m,\biota)$  (where $\m$ is the multiplication and $\biota$ is the unit), one first defines the \emph{Eilenberg-Moore category} $\EM\Xx\A$ of $\A$-\emph{algebras}, whose objects are  pairs $(X,\rho^X)$, where $X$ is an object in $\Xx$ and $\rho^X:AX\to X$ is a morphism in $\Xx$ such that $\rho^X\circ A\rho^X= \rho^X\circ \m X$ and $ \rho^X\circ\biota X=X$. There is a pair of adjoint functors $(\fr A:\Xx\to \EM\Xx\A, U_A: \EM\Xx\A\to \Xx)$, where $U_A$ is a \emph{forgetful functor} and $\fr A$ is the  \emph{induction} or {\em free algebra functor}, for all objects $X\in\Xx$ defined by $\fr AX=(AX,\m X)$. Similarly, one defines the category $\EMC\Yy\C$ of \emph{coalgebras} over $\C$, and obtains an adjoint pair $(U^C:\EMC\Yy\C\to \Yy, \frc C:\Yy\to \EMC\Yy\C)$, where $\frc C$ is the induction or free coalgebra functor and $U^C$ is the forgetful functor. The original and constructed adjunctions are related by the \emph{comparision}  $K:\Yy\to \EM\Xx\A$ (or $K':\Xx\to \EMC\Yy\C$ in the comonad case). For any $Y\in \Yy$, the comparision functor is given by 
$KY=(RY,R\varepsilon Y)$. The functor $R$ is said to be \emph{monadic} if $K$ is an equivalence of categories. Similarly $L$ is said to be \emph{comonadic} if $K'$ is an equivalence of categories. \emph{Beck's Theorem} \cite{Bec:tri} provides one with necessary and sufficient conditions for the functors $R$ and $L$ to be monadic or comonadic respectively. For more detailed and comprehensive study of matters described in this section we refer to \cite{BarWel:ttt}. 

\subsection{The category of double adjunctions}\label{sec.adj}
In this paper, rather than looking at \emph{one} adjunction, we consider \emph{two} adjunctions  with right adjoints operating on a common category. More precisely, let $\Xx$ and $\Yy$ be two categories.
Following \cite{BohMen:pretor}, the category $\Adj(\Xx,\Yy)$ is defined as follows.
An object in $\Adj(\Xx,\Yy)$ is a pentuple (or a triple) $\tt=(\Tt,(L_A,R_A),(L_B,R_B))$, where $\Tt$ is a category and
$(L_A:\Xx\to\Tt,R_A:\Tt\to\Xx)$ and $(L_B:\Yy\to\Tt,R_B:\Tt\to\Yy)$ are adjunctions whose units and counits are denoted respectively by $\eta^A$, $\eta^B$ and $\varepsilon^A$, $\varepsilon^B$.
A morphism $F:\tt=(\Tt,(L_A,R_A),(L_B,R_B))\to \tt'=(\Tt',(L_A',R_A'),(L_B',R_B'))$ is a functor $F:\Tt'\to \Tt$ such that $R_AF=R'_A$ and $R_BF=R'_B$. 
In other words, $\Adj(\Xx,\Yy)$ is a full subcategory of the category of {\em spans} $\Span(\Xx,\Yy)$ which consists of all spans $\Xx \leftarrow \Tt \rightarrow \Yy$ such that the functors $\Tt\to \Xx$ and $\Tt\to \Yy$ have left adjoints.

To a morphism $F$ in $\Adj(\Xx,\Yy)$, one associates two natural transformations
\begin{equation}\label{eq:a,b}
a:=(\varepsilon^AFL'_A)\circ(L_A\eta'^A):L_A\to FL'_A\ , \qquad 
b:=(\varepsilon^BFL'_B)\circ(L_B\eta'^B):L_B\to FL'_B\ ,
\end{equation}
 which satisfy the following compatibility conditions
\begin{eqnarray}
(R_Aa)\circ \eta^A=\eta'^A, &\qquad& (R_Bb)\circ\eta^B=\eta'^B, \label{adjmorph1}\\
(F\varepsilon'^A)\circ(aR'_A)=\varepsilon^AF, &\qquad& (F\varepsilon'^B)\circ(bR'_B)=\varepsilon^BF. \label{adjmorph2}
\end{eqnarray}

\subsection {The category of Morita contexts}\label{sec.morita}
Consider  two categories $\Xx$ and $\Yy$. Let $\A=(A,\m^A,\biota^A)$ be a monad on $\Xx$, $\B=(B,\m^B,\biota^B)$ be a monad on $\Yy$. An $\A$-$\B$ \emph{bialgebra}  \emph{functor} $\T=(T,\lambda,\rho)$, is a functor $T:\Yy\to\Xx$ equipped with two natural transformations $\rho:TB\to T$ and $\lambda:AT\to T$ such that
\begin{eqnarray}
&\xymatrix{TBB \ar[rr]^-{T\m^B}\ar[d]_{\rho B} && TB \ar[d]^{\rho} \\
TB \ar[rr]^-{\rho } && T\ ,} \qquad 
\xymatrix{TB \ar[r]^\rho & T\\
 & T \ar[u]_{=} \ar[ul]^{T\biota^B} \ ,} \label{bialg1}\\
&\xymatrix{AAT \ar[rr]^-{\m^AT}\ar[d]_{A\lambda} && AT \ar[d]^{\lambda} \\
AT \ar[rr]^-{\lambda } && T\ ,} \qquad 
\xymatrix{AT \ar[r]^\lambda & T \\
 & T \ar[u]_{=} \ar[ul]^{\biota^A T} \ ,} \qquad 
 \xymatrix{ATB \ar[rr]^-{\lambda B}\ar[d]_{A\rho} && TB \ar[d]^{\rho} \\
AT \ar[rr]^-{\lambda  } && T\ .}\label{bialg2}
\end{eqnarray}
A bialgebra morphism $\phi:\T\to \T'$ between two $\A$-$\B$ bialgebra functors is a natural transformation that satisfies the following conditions
\[
\xymatrix{
TB \ar[rr]^-{\rho} \ar[d]_-{\phi B} && T \ar[d]^\phi \\
T'B \ar[rr]_-{\rho'} && T' \ ,
}
\qquad
\xymatrix{
AT \ar[rr]^-{\lambda} \ar[d]_-{A\phi} && T\ar[d]^\phi \\
AT' \ar[rr]_-{\lambda'} && T' \ .
}
\]
An $\A$-$\B$ bialgebra functor $\T$ induces a functor $\Yy\to \EM\Xx\A$, which is denoted again by $T$ and is defined by $TY=(TY,\lambda Y)$, for all $Y\in \Yy$.

A \emph{Morita context} on $\Xx$ and $\Yy$ is a sextuple $\TT=(\A,\B,\T,\bhT,\ev,\hev)$, that consists of a monad $\A=(A,\m^A,\biota^A)$ on $\Xx$, a monad $\B=(B,\m^B,\biota^B)$ on $\Yy$, an $\A$-$\B$ bialgebra functor $\T$, a $\B$-$\A$ bialgebra functor $\bhT$ and natural transformations $\ev:T\hT\to A$ and $\hev:\hT T\to B$. These are required to satisfy the following conditions: $\ev$ is an $\A$-$\A$ bialgebra morphism, $\hev$ is a $\B$-$\B$ bialgebra morphism, and the following diagrams commute
\begin{eqnarray}
 \xymatrix{T\hT T \ar[rr]^-{T\hev}\ar[d]_{\ev T} && TB \ar[d]^{\rho} \\
AT \ar[rr]^-{\lambda } && T\ ,} &\qquad&
 \xymatrix{\hT T \hT \ar[rr]^-{\hev \hT}\ar[d]_{\hT \ev} && B\hT \ar[d]^{\hlambda} \\
\hT A\ar[rr]^-{\hrho } && \hT \ ,} \label{morita} \\
\xymatrix{
\hT AT \ar[rr]^-{\hT \lambda} \ar[d]_-{\hrho T} && \hT T \ar[d]^\hev \\
\hT T \ar[rr]_-{\hev} && B \ ,
} &\qquad&
\xymatrix{
TB\hT \ar[rr]^-{\rho\hT} \ar[d]_{T\hlambda} && T\hT \ar[d]^\ev \\
T\hT \ar[rr]_-\ev && A\ . \label{balanced}
}
\end{eqnarray}
Diagrams \eqref{balanced} mean that $\ev$ is {\em $\B$-balanced} and $\hev$ is {\em $\A$-balanced}. A pair of bialgebras $\T$, $\bhT$ that satisfy all the conditions in this section {\em except for} diagrams \eqref{balanced} is termed a {\em pair of formally dual bialgebras} or a {\em pre-Morita context}. Thus a Morita context is a balanced pair of formally dual bialgebras.

\begin{remark}\label{rem.Moritaring}
Let $k_A$ and $k_B$ be unital associative rings. Set $\Xx$ to be the category of left $k_A$-modules and $\Yy$ the category of left $k_B$-modules. For a $k_A$-ring $A$ (i.e.\ a ring map $k_A\to A$) with multiplication $\mu_A$ and unit $\iota_A$ and a $k_B$-ring $B$ with multiplication $\mu_B$ and unit $\iota_B$ consider monads  $\A=(A\ot_{k_A}-, \mu_A\ot_{k_A}-,\iota_A\ot_{k_A}-)$ and $\B=(B\ot_{k_B}-,\mu_B\ot_{k_B}-,\iota_B\ot_{k_B}-)$.  Then the definition of a Morita context with monads $\A$ and $\B$ coincides  with the classical definition of a (ring-theoretic) Morita context between rings $A$ and $B$; see Section~\ref{sec.morita} for a more detailed study of this example. Note however that the definition of a categorical Morita context introduced in this paper differs from the notion of a {\em wide Morita context} \cite{CIGT:wide}, which also gives a categorical albeit different interpretation of classical Morita contexts. The connection between categorical Morita contexts and wide Morita contexts is discussed in Section~\ref{sec.Moritath}.
\end{remark}

A morphism $\phi:\TT=(\A,\B,\T,\bhT,\ev,\hev)\to \TT'=(\A',\B',\T',\bhT',\ev',\hev')$ between Morita contexts on $\Xx$ and $\Yy$ is a quadruple $(\phi^1,\phi^2,\phi^3,\phi^4)$ defined as follows. First,  $\phi^1:\A\to\A'$ and $\phi^2:\B\to\B'$ are monad morphisms, i.e.\ $\phi^1:A\to A'$ and $\phi^2:B\to B'$ are natural transformation such that 
\begin{eqnarray}
\xymatrix{
AA \ar[rr]^-{\m^A} \ar[d]_{\phi^1\phi^1} && A \ar[d]^{\phi^1} \\
A'A' \ar[rr]_-{\m^{A'}} && A' \ ,
}\qquad
\xymatrix{
\id_\Xx \ar[drr]_{\biota^{A'}} \ar[rr]^{\biota^A} && A \ar[d]^{\phi^1}\\
&& A' \ ,
}\label{monadmorphA}
\\
\xymatrix{
BB \ar[rr]^-{\m^B} \ar[d]_{\phi^2\phi^2} && B \ar[d]^{\phi^2} \\
B'B' \ar[rr]_-{\m^{B'}} && B' \ ,
}\qquad
\xymatrix{
\id_\Yy \ar[drr]_{\biota^{B'}} \ar[rr]^{\biota^B} && B \ar[d]^{\phi^2}\\
&& B' \ , 
}\label{monadmorphB}
\end{eqnarray}
where the shorthand notation $\phi^1\phi^1 := A'\phi^1\circ \phi^1A = \phi^1 A'\circ A\phi^1$ etc.\  for the Godement product of natural transformations is used. These make $\T'$ an  $\A$-$\B$ bialgebra functor and $\bhT'$ a  $\B$-$\A$ bialgebra functor with structures given by
\begin{eqnarray*}
\xymatrix{
T'B \ar[rr]^-{T'\phi^2} && T'B' \ar[rr]^-{\rho'} && T' \ ,
} &\qquad &
\xymatrix{
AT' \ar[rr]^-{\phi^1T'} && A'T' \ar[rr]^-{\lambda'} && T' \ ,
}\\
\xymatrix{
\hT'A \ar[rr]^-{\hT'\phi^1} && \hT'A' \ar[rr]^-{\hrho'} && \hT' \ ,
} &\qquad&
\xymatrix{
B\hT' \ar[rr]^-{\phi^2\hT'} && B'\hT' \ar[rr]^-{\hlambda'} && \hT' \ . 
}
\end{eqnarray*}
Second, $\phi^3:\T\to \T'$ is a morphism of $\A$-$\B$ bialgebras and $\phi^4:\bhT\to \bhT'$ is a morphism of $\B$-$\A$ bialgebras. These morphisms are required to satify the following compatibility conditions
\begin{equation}\label{moritamorph}
\xymatrix{
T\hT \ar[rr]^-{\ev} \ar[d]_{\phi^3\phi^4} && A \ar[d]^{\phi^1} \\
T'\hT' \ar[rr]_-{\ev'} && A' \ ,
}\qquad
\xymatrix{
\hT T \ar[rr]^-{\hev} \ar[d]_{\phi^4\phi^3} && B \ar[d]^{\phi^2} \\
\hT' T' \ar[rr]_-{\hev'} && B' \ .
}
\end{equation}

This completes the definition of  the category $\Mor(\Xx,\Yy)$ of Morita contexts on $\Xx$ and $\Yy$.

\section{A Beck-type theorem for a Morita context}\label{sec.Beck}
The aim of this section is to analyse the relationship between double adjunctions described in Section~\ref{sec.adj} and Morita contexts defined in Section~\ref{sec.morita}. The same notation and conventions as in Section~\ref{adjMor} are used.

\subsection {From double adjunctions to Morita contexts}\label{adj.to.mor}

Fix categories $\Xx$ and $\Yy$. The aim of this section is to construct a functor
$\Upsilon$ from the category of double adjunctions $\Adj(\Xx,\Yy)$ to the category of Morita contexts $\Mor(\Xx,\Yy).$
Take any object $\tt=(\Tt,(L_A,R_A),(L_B,R_B))$ in $\Adj(\Xx,\Yy)$ and define an object
$$\Upsilon(\tt)=\TT:=(\A,\B,\T,\bhT,\ev,\hev)$$ 
in $\Mor(\Xx, \Yy)$ as follows. 

The adjunction $(L_A,R_A)$ defines a monad $\A:=(A:=R_AL_A,\m^A:=R_A\varepsilon^AL_A,\eta^A)$ on $\Xx$ and the adjunction $(L_B,R_B)$ defines a monad $\B:=(B:=R_BL_B,\m^B:=R_A\varepsilon^BL_B,\eta^B)$ on $\Yy$; see Section~\ref{sec.eil-moo}.
Set $T:=R_AL_B:\Yy\to \Xx$ and view it as an $\A$-$\B$ bialgebra functor by  $R_A\varepsilon^AL_B: AT=R_AL_AR_AL_B\to R_AL_B = T$ and $R_A\varepsilon^BL_B : TB=R_AL_BR_BL_B\to R_AL_B = T$, i.e.\ $\T=(R_AL_B,R_A\varepsilon^AL_B,R_A\varepsilon^BL_B)$. To check that $\T$ is  a bialgebra functor one can proceed as follows: the associativity conditions (of left $\A$-action, right $\B$-action and mixed associativity) are a straightforward consequence of naturality of the counits. The unitality of left and right actions follows by the triangular identities for units and counits of  adjunctions $(L_A,R_A)$ and $(L_B,R_B)$. 
The $\B$-$\A$ bialgebra functor $\bhT:=(R_BL_A,R_B\varepsilon^BL_A,R_B\varepsilon^AL_A)$ is defined similarly. 
Finally define natural transformations
\begin{eqnarray*}
\xymatrix{\ev: T\hT = R_AL_BR_BL_A \ar[rrr]^-{R_A\varepsilon^BL_A} &&& R_AL_A = A \ ,}\\
\xymatrix{\hev: \hT T= R_BL_AR_AL_B \ar[rrr]^-{R_B\varepsilon^AL_B} &&& R_BL_B = B \ .} 
\end{eqnarray*}
All compatibility diagrams \eqref{morita}-\eqref{balanced} follow (trivially) by the naturality of the counits $\varepsilon^A$ and $\varepsilon^B$.

For any  morphism $F:\tt\to \tt'$  in $\Adj(\Xx,\Yy)$, the corresponding morphism of Morita contexts,
$$\Upsilon(F)=(\phi^1,\phi^2,\phi^3,\phi^4) : \Upsilon(\tt)\to \Upsilon(\tt'), $$
is defined as follows. The monad morphisms are:
\begin{eqnarray*}
\xymatrix{
\phi^1:R_AL_A \ar[rr]^{R_Aa} && R_AFL'_A =R'_AL'_A \ ,
} \\
\xymatrix{
\phi^2:R_BL_B \ar[rr]^{R_Bb} && R_BFL'_B =R'_BL'_B \ ,
}
\end{eqnarray*}
where $a$ and $b$ are defined by equations \eqref{eq:a,b} in Section~\ref{sec.adj}.
Equations \eqref{adjmorph1} express exactly that $\phi^1$ and $\phi^2$ satisfy the second diagrams in \eqref{monadmorphA} and \eqref{monadmorphB}, respectively. That $\phi^1$ and $\phi^2$ satisfy the first diagrams in  \eqref{monadmorphA} and \eqref{monadmorphB} follows by \eqref{adjmorph2} and by the naturality.

The morphisms of bialgebras are:
\begin{eqnarray*}
\xymatrix{
\phi^3:T=R_AL_B \ar[rr]^{R_Ab} && R_AFL'_B =R'_AL'_B=T' \ ,
}\\
\xymatrix{
\phi^4:\hT=R_BL_A \ar[rr]^{R_Ba} && R_BFL'_A =R'_BL'_A=\hT' \ .
}
\end{eqnarray*}
That $\phi^3$ and $\phi^4$ are bialgebra morphisms satisfying compatibility conditions \eqref{moritamorph} follows by equations \eqref{adjmorph2} and by the naturality (the argument is very similar to the one used for checking that $\phi^1$ and $\phi^2$ preserve  multiplications). 

\subsection {The Eilenberg-Moore category of a Morita context}\label{sec.eil-moo.mor}
Before we can make a converse construction for the functor $\Upsilon$ introduced in Section~\ref{adj.to.mor}, we need to define a category of representations for a Morita context. This construction is similar to that of the Eilenberg-Moore category of algebras for a monad; see Section~\ref{sec.eil-moo}.

Let $\TT=(\A,\B,\T,\bhT,\ev,\hev)$ be an object in $\Mor(\Xx,\Yy)$. The {\em Eilenberg-Moore category  associated to $\TT$}, $\EMM\TT$, is defined as follows. Objects in $\EMM\TT$ are sextuples (or quadruples) $\X=((X,\rho^X),(Y,\rho^Y),v,w)$, where $(X,\rho^X)\in\EM\Xx\A$ is an algebra for the monad $\A$, $(Y,\rho^Y)\in\EM\Yy\B$ is an algebra for the monad $\B$, $v:TY\to X$ is a morphism in $\EM\Xx\A$ and $w:\hT X\to Y$ is a morphism in $\EM\Yy\B$ (i.e. 
\begin{equation}
\xymatrix{
ATY \ar[rr]^-{Av} \ar[d]_{\lambda Y} && AX \ar[d]^{\rho^X} \\
TY \ar[rr]_-{v} && X \ ,
}\qquad
\xymatrix{
B\hT X \ar[rr]^-{Bw} \ar[d]_{\hlambda X} && BY \ar[d]^{\rho^Y}\\
\hT X \ar[rr]_-w && Y \ ),
}\label{vwlin}
\end{equation}
satisfying the following compatibility conditions
\begin{eqnarray}
\xymatrix{
T\hT X \ar[rr]^-{Tw} \ar[d]_{\ev X} && TY \ar[d]^{v} \\
AX \ar[rr]_-{\rho^X} && X \ ,
} \qquad
\xymatrix{
\hT T Y \ar[rr]^-{\hT v} \ar[d]_{\hev Y} && \hT X \ar[d]^{w} \\
BY \ar[rr]_-{\rho^Y} && Y \ ,
}\label{vwass1}\\
\xymatrix{
\hT AX \ar[rr]^-{\hT \rho^X} \ar[d]_{\hrho X} && \hT X \ar[d]^w \\
\hT X \ar[rr]_-{w} && Y \ ,
}\qquad
\xymatrix{
TBY \ar[rr]^-{T\rho^Y} \ar[d]_{\rho Y} && TY \ar[d]^{v}\\
TY \ar[rr]_-{v} && X \ .
}\label{vwass2}
\end{eqnarray}
All these diagrams can be understood as generalised mixed associativity conditions.

A morphism $\X\to \X'$ in $\EMM\TT$ is  a couple $(f,g)$, where $f: X\to X'$ is a morphism in $\EM\Xx\A$ and $g:Y\to Y'$ is a morphism in $\EM\Yy\B$ (i.e.\ 
\begin{equation}
\xymatrix{
AX \ar[rr]^{Af} \ar[d]_{\rho^X} && AX' \ar[d]^{\rho^{X'}} \\
X \ar[rr]_f && X' \ , 
}\qquad
\xymatrix{
BY \ar[rr]^{Bg} \ar[d]_{\rho^Y} && BY' \ar[d]^{\rho^{Y'}} \\
Y \ar[rr]_g && Y' \ ),
}\label{fglin}
\end{equation}
such that
\begin{equation}
\xymatrix{
\hT X \ar[rr]^-{\hT f} \ar[d]_{w} && \hT X' \ar[d]^{w'} \\
Y \ar[rr]_-{g} && Y' \ ,
} \qquad
\xymatrix{
TY \ar[rr]^-{T g} \ar[d]_v && TY' \ar[d]^{v'} \\
X \ar[rr]_-{f} && X'\ .
}\label{fgcomp}
\end{equation}
This completes the construction of the Eilenberg-Moore category of a Morita context. 

\subsection {From Morita contexts to double adjunctions}\label{mor.to.adj}
In this section  a functor $\Gamma$ from the category of Morita contexts on $\Xx$ and $\Yy$, $\Mor(\Xx,\Yy)$, to the category of double adjunctions $\Adj(\Xx,\Yy)$ is constructed.

 For a Morita context $\TT = (\A, \B, \T, \bhT,\ev, \hev)\in\Mor(\Xx,\Yy)$, the double adjunction 
 $$\Gamma(\TT)=\tt=(\EMM\TT,(G_A,U_A),(G_B,U_B)),
 $$ 
 is defined as follows. $\EMM\TT$ is the Eilenberg-Moore category for the Morita context $\TT$ as defined in Section~\ref{sec.eil-moo.mor}. The functors $U_A:\EMM\TT\to \Xx$ and $U_B:\EMM\TT\to \Yy$ are the  forgetful functors (i.e.\ $U_A\X=X$ and $U_B\X=Y$ for all objects $\X=((X,\rho^X),(Y,\rho^Y),v,w)$ in $\EMM\TT$). The definition of the functors $G_A$ and $G_B$ is slightly more involved. For any $X\in\Xx$ and $Y\in\Yy$, define 
\begin{eqnarray*}
G_AX&=&((AX,\m^AX),(\hT X,\hlambda X),\ev X, \hrho X) ,\\
G_BY&=&((TY,\lambda Y),(BY,\m^BY),\rho Y,\hev Y).
\end{eqnarray*}
$(AX,\m^AX)$ is simply the free $\A$-algebra on $X$ (see Section~\ref{sec.eil-moo}), hence it is an object of $\EM\Xx\A$. From the discussion in Section~\ref{sec.morita}  we know that $(\hT X,\hlambda X)\in \EM\Yy\B$, with  a $\B$-algebra structure induced by the $\B$-$\A$ bialgebra functor $\hT$. To complete the check that $G_AX$ is an object of $\EMM \TT$ it remains to verify whether the maps $\ev X:T\hT X\to AX$ and $\rho X: \hT AX\to \hT X$ satisfy all needed compatibility conditions. The left hand side of \eqref{vwlin} expresses that $\ev$ is left $\A$-linear, the left hand side of \eqref{vwass1} that $\ev$ is right $\A$-linear. The right hand side of \eqref{vwlin}  follows by the mixed associativity of the $\B$-$\A$ bialgebra $\bhT$ (see the last diagram of \eqref{bialg2}). The right hand side of \eqref{vwass1} is the second Morita identity of the maps $\ev$ and $\hev$; see the right hand side of \eqref{morita}. The left hand side of \eqref{vwass2} follows again from the properties of $\bhT$ as a $\B$-$\A$ bialgebra, in particular by the associativity of its right $\A$-action; compare with the first diagram in \eqref{bialg1}. Finally, the right hand side of \eqref{vwass2} is an application of the fact that $\ev$ is $\B$-balanced, which is expressed in the right hand side of \eqref{balanced}. We conclude that $G_A$ (and, by symmetric arguments, also $G_B$) is well-defined on objects. For a morphism $f:X\to X'$ in $\Xx$, define $G_Af=(Af,\hT f)$ and similarly $G_Bg=(Tg,Bg)$ for any morphism $g$ in $\Yy$. Verification that $G_Af$ and $G_Bg$ are well defined is  very simple and left to the reader.

\begin{lemma} \label{lem.gu}
$(\EMM\TT, (G_A,U_A),(G_B,U_B))$ is a double adjunction. 
\end{lemma}
\begin{proof} We  construct units $\nu^A, \nu^B$  and counits $\zeta^A, \zeta^B$ of adjunctions. 
For any objects $X\in\Xx$ and $Y\in\Yy$,  $\nu^A$, $\nu^B$ are defined as morphisms in $\Xx$ and $\Yy$ respectively,
$$
\nu^A X=\eta^AX :X\to AX, \qquad
\nu^BY=\eta^BY :Y\to BY.
$$
For any $\X\in\EMM\TT$, the counits $\zeta^A$, $\zeta^B$ are given by the following morphisms in $\EMM\TT$,
\begin{eqnarray*}
\zeta^A_\X=(f^A\X,g^A\X):G_AU_A\X&\to& \X,\\
((AX,\m^AX),(\hT X,\hlambda X),\ev X, \hrho X)&\to& ((X,\rho^X),(Y,\rho^Y),v,w),\\
 f^A \X=\rho^X:AX\to X,&\ & g^A \X=w:\hT X\to Y,\\
\zeta^B \X=(f^B \X,g^B \X):G_BU_B\X&\to&\X,\\
((TY,\lambda Y),(BY,\m^BY),\rho Y,\hev Y)&\to& ((X,\rho^X),(Y,\rho^Y),v,w),\\
 f^B X=v:TY\to X,&\ &g^B \X=\rho^Y:BY\to Y.
\end{eqnarray*}
To check that $\zeta^A \X$ is  a morphism in $\EMM\TT$, one has to verify that diagrams \eqref{fglin} and \eqref{fgcomp} commute. The left hand side of \eqref{fglin} holds, since $f^A\X=\rho^X$ is canonically a morphism in $\EM\Xx\A$, the right hand side holds since $g^A\X=w$ is a morphism in $\EM\Yy\B$ by definition. The left hand side of \eqref{fgcomp} is exactly the left hand side of \eqref{vwass2}, and the right hand side of \eqref{fgcomp} is precisely the left hand side of \eqref{vwass1}. Similarly one checks that $\zeta^B\X$ is a morphism in $\EMM\TT$. 

Now take any object $X\in\Xx$. The first triangular identity translates to the following diagram 
\[
\xymatrix{
G_AX \ar[rrd]^-{(A\eta^AX,\hT\eta^AX)} \ar@{=}[dd]\\
&& G_AU_AG_AX \ar[dll]^-{(\m^AX,\hrho X)} ,\\
G_AX
}
\]
which commutes by the unit properties of the monad $\A$ and the bialgebra functor $\hT$ (i.e. $\m^A\circ A\eta^A= A$ and $\hrho\circ \hT \eta^A=\hT$). For the second triangular identity,  take any $\X\in\EMM{\TT}$ and consider the diagram
\[
\xymatrix{
U_A\X=X \ar[drr]^{\eta^AX}\ar@{=}[dd] \\
&& U_AG_AU_A\X=AX \ar[dll]^-{\rho^X}, \\
U_A\X=X
}
\]
which commutes by the unit property of the $\A$-algebra $(X,\rho^X)$. In the same way one verifies that $(G_B,U_B)$ is an adjoint pair.
\end{proof}

Let $\phi=(\phi^1,\phi^2,\phi^3,\phi^4):\TT\to\TT'$ be a morphism of Morita contexts. We need to construct a morphism $\Gamma(\phi):\Gamma(\TT)\to\Gamma(\TT')$ in $\Adj(\Xx,\Yy)$, i.e.\ a functor $F=\Gamma(\phi):\EMM{\TT'}\to \EMM{\TT}$ such that $U_A F= U'_A$ and $U_B F= U'_B$. It is well-known that a morphism of monads $\phi^1:\A\to \A'$ induces a functor $\Phi_1:\EM \Xx{\A'}\to \EM\Xx\A$, given by $\Phi_1(X,\rho^{X})=(X,\rho^X\circ \phi^1{X})$, for all objects $(X,\rho^X)\in \EM \Xx{\A'}$, and $\Phi_1f=f$ for all morphisms $f$ in $\EM \Xx{\A'}$ (the naturality of $\phi^1$ implies  that $f$ is indeed a morphism in $\EM\Xx\A$). Similarly, a morphism of monads $\phi^2:\B\to \B'$ induces a functor $\Phi_2:\EM \Yy{\B'}\to\EM\Yy\B$. Take any $\X=((X,\rho^X),(Y,\rho^Y),v,w)\in\EMM{\TT'}$ and  define 
$$F\X=((X,\rho^X\circ \phi^1{X}),(Y,\rho^Y\circ \phi^2Y),v\circ\phi^3Y,w\circ\phi^4X).$$
To check diagrams \eqref{vwlin}, \eqref{vwass1}, \eqref{vwass2}, one has to rely on the naturality of $\phi^1,\phi^2,\phi^3,\phi^4$, and their properties as morphisms of monads and bialgebras.

For a morphism $(f,g):\X\to\X'$ in $\EMM{\TT'}$,  define $F(f,g)=(f,g)$. Then, by construction (or by naturality of $\phi^1$ and $\phi^2$), $f$ is a morphism in $\EM\Xx\A$ and $g$ is a morphism in $\EM\Yy\B$. Diagram \eqref{fgcomp} for $\F(f,g)$ to be a morphism in $\EMM{\TT}$ follows by the naturality of $\phi^3$ and $\phi^4$, combined with the corresponding diagram for  $(f,g)$ as a morphism in $\EMM{\TT'}$.

Finally, the construction of $F$ immediately implies that,  for any $\X\in\EMM{\TT'}$, $U_A F\X=U'_A\X=X$ and $U_B F\X=U'_B\X=Y$.
This completes the construction of a functor $\Gamma: \Mor(\Xx,\Yy)\to \Adj(\Xx,\Yy)$.

\subsection {Every Morita context arises from a double adjunction}\label{sec.mor.from.double} Here we prove that starting with a Morita context and performing  subsequent constructions of a double adjunction and a Morita context gives back the original Morita context, i.e.\ we prove the following
\begin{lemma}
The composite functor $\Upsilon\Gamma$ is the identity functor on $\Mor(\Xx,\Yy)$.
\end{lemma}
\begin{proof}
The computation that, for any Morita context $\TT\in\Mor(\Xx,\Yy)$,  $\Upsilon\Gamma\TT = \TT$ is easy and left to the reader. Consider two Morita contexts  $\TT=(\A,\B,\T,\bhT,\ev,\hev)$, $\TT'=(\A',\B',\T',\bhT',\ev',\hev')$ and a morphism $\phi = (\phi^1,\phi^2,\phi^3,\phi^4) : \TT \to \TT'$. Write
$$
(\phi^1_{\Gamma(\phi)}, \phi^2_{\Gamma(\phi)}, \phi^3_{\Gamma(\phi)}, \phi^3_{\Gamma(\phi)}) := \Upsilon\Gamma\phi.
$$
In view of the definition of functor $\Upsilon$ in Section~\ref{adj.to.mor}, to compute the $\phi^i_{\Gamma(\phi)}$ one first needs to compute natural transformations $a$, $b$ (see equations \eqref{eq:a,b} in Section~\ref{sec.adj}) corresponding to double adjunctions $\Gamma\TT = (\EMM{\TT}, (G_A,U_A), (G_B,U_B))$  and $\Gamma\TT'= (\EMM{\TT'}, (G'_A,U'_A), (G'_B,U'_B))$; see Section~\ref{mor.to.adj}. These are given by
$$
a = \zeta^A\Gamma(\phi)G'_A\circ G_A\nu'^A, \qquad b = \zeta^B\Gamma(\phi)G'_B\circ G_B\nu'^B
$$
where $\zeta^A$ is the counit of adjunction $(G_A,U_A)$, $\zeta^B$ is the counit of adjunction $(G_B,U_B)$, $\nu'^A$ is the unit of adjunction $(G'_A,U'_A)$ and $\nu'^B$ is the unit of adjunction $(G'_B,U'_B)$. Since, for all $\X=((X,\rho^X),(Y,\rho^Y),v,w)\in\EMM{\TT'}$, 
$$\Gamma(\phi)(\X)=((X,\rho^X\circ \phi^1{X}),(Y,\rho^Y\circ \phi^2Y),v\circ\phi^3Y,w\circ\phi^4X),
$$
we obtain
$$
\zeta^A\Gamma(\phi)G'_A = (\m^{A'} \circ \phi^1 A', \hrho' \circ \phi^4 A'), \quad G_A\nu'^A = (A\biota^{A'}, \hT\biota^{A'}).
$$
Thus 
$$
\phi^1_{\Gamma(\phi)} = U_A a = \m^{A'}\circ \phi^1 A'\circ A\biota^{A'} = \m^{A'}\circ A'\biota^{A'}\circ \phi^1 = \phi^1,
$$
where the second equality follows by the naturality of $\phi^1$, while the third one is a consequence of the unitality of a monad. Similarly, 
$$
\phi^4_{\Gamma(\phi)} = U_B a = \hrho'\circ \phi^4 A'\circ \hT\biota^{A'} = \hrho'\circ \hT'\biota^{A'}\circ \phi^4 = \phi^4.
$$
The identities  $\phi^2_{\Gamma(\phi)} = \phi^2$ and $\phi^3_{\Gamma(\phi)}=\phi^3$  are obtained by symmetric calculations.
\end{proof}

The just computed identification thus defines a natural transformation (the identity transformation)
$$
\lambda: \id_{\Mor(\Xx,\Yy)} \to \Upsilon\Gamma.
$$

\subsection {The comparison functor}\label{sec.comp}

Consider  a double adjunction on  categories $\Xx$ and $\Yy$, i.e.\ an object  $\tt=(\Tt,(L_A,R_A),(L_B,R_B))$ in $\Adj(\Xx,\Yy)$. Let $\Upsilon\tt=\TT$ be the associated Morita context on $\Xx$ and $\Yy$, and consider $\EMM{\TT}$, the Eilenberg-Moore category of representations of $\TT$. In this section we construct a \emph{comparison functor }
$$K:\Tt\to \EMM{\TT}.$$

For any object $Z\in\Tt$, define
$$K(Z)=((R_AZ,R_A\varepsilon^A Z),(R_BZ,R_B\varepsilon^B Z),R_A\varepsilon^B Z,R_B\varepsilon^A Z).$$
The first two components in $K(Z)$, that is, $(R_AZ,R_A\varepsilon^A Z)\in\EM\Xx\A$ and $(R_BZ,R_B\varepsilon^B Z)\in\EM\Yy\B$ are an application of the comparison functors $K_A:\Tt\to \EM\Xx\A$ and $K_B:\Tt\to \EM\Yy\B$, corresponding to the adjunctions $(L_A,R_A)$ and $(L_B,R_B)$ respectively (see Section~\ref{sec.eil-moo}). Obviously,
\begin{eqnarray*}
R_A\varepsilon^B Z:R_AL_BR_BZ\to R_AZ & {\rm and} & R_B\varepsilon^A Z:R_BL_AL_BZ\to R_BZ
\end{eqnarray*}
are well-defined. They satisfy conditions \eqref{vwlin}, \eqref{vwass1} and \eqref{vwass2} by the naturality of counits.

For a morphism $f:Z\to Z'$ in $\Tt$, define 
$$K(f)=(R_Af,R_Bf).$$
In view of the definition of the comparison functors $K_A$ and $K_B$, it is clear that $R_Af=K_Af$ and $R_Bf=K_Bf$, so $R_Af$ and $R_Bf$ are morphisms in $\EM\Xx\A$ and $\EM\Yy\B$ respectively. Diagrams \eqref{fgcomp} follow by the naturality of $\varepsilon^A$ and $\varepsilon^B$, respectively.

\begin{definition}\label{def.moritable}
Let  $\tt=(\Tt,(L_A,R_A),(L_B,R_B))$ be an object in $\Adj(\Xx,\Yy)$. The pair $(R_A,R_B)$ is said to be  \emph{moritable} if and only if the comparison functor $K$ is an equivalence of categories.
\end{definition}

\begin{proposition}\label{prop.adjoint}
$(\Gamma, \Upsilon)$ is an adjoint pair and $\Gamma$ is a full and faithful functor.
\end{proposition}
\begin{proof}
Note that the definition of the comparison functor $K$ immediately implies that $K$ is a morphism in $\Adj(\Xx,\Yy)$. Furthermore, $K$ can be defined for any $\tt\in\Adj(\Xx,\Yy)$. We claim that  the assignment $\tt \mapsto (K:\Tt\to \EMM{\TT})$ induces a natural transformation 
$$\kappa:\Gamma\Upsilon \to \id_{\Adj(\Xx,\Yy)}.$$ 
Take double adjunctions  $\tt=(\Tt,(L_A,R_A),(L_B,R_B))$, $\tt'=(\Tt',(L'_A,R'_A),(L'_B,R'_B))$ and a functor $F:\Tt \to \Tt'$ such that $R_AF = R_A'$ and $R_BF=R'_B$ (in other words, take a morphism in $\Adj(\Xx,\Yy)$). Let $K: \Tt \to \EMM{\Upsilon\tt}$ and $K': \Tt' \to \EMM{\Upsilon\tt'}$ be the associated comparison functors. The naturality of $\kappa$ is  equivalent to the commutativity of the following diagram
$$
\xymatrix{ \Tt' \ar[rr]^-F\ar[d]_{K'} && \Tt \ar[d]^K \\  \EMM{\Upsilon\tt'} \ar[rr]^-{\Gamma\Upsilon (F)} &&  \EMM{\Upsilon\tt} .}
$$
For any object $Z$ of $\Tt'$, 
$$
KFZ = ((R_AFZ,R_A\varepsilon^AFZ),(R_BZ,R_B\varepsilon^BFZ),R_A\varepsilon^BFZ,R_B\varepsilon^AFZ),
$$
and 
\begin{eqnarray*}
(\Gamma\Upsilon(F) K')(Z)  &=& ((R'_A Z, R'_A\eps'^AZ\circ R_AaR'_AZ), (R'_B Z, R'_B\eps'^BZ\circ R_BbR'_BZ),\\
&& R'_A\eps'^BZ\circ R_AbR'_BZ, R'_B\eps'^AZ\circ R_BaR'_AZ),
\end{eqnarray*}
where $a$ and $b$ are natural transformations  \eqref{eq:a,b} associated to a morphism of double adjunctions $F$. The equalities $R_AF = R_A'$ and $R_BF=R'_B$ together with equations \eqref{adjmorph2} yield the required equality
$K F = \Gamma\Upsilon(F) K'$.

Let $\lambda: \id_{\Mor(\Xx,\Yy)} \to \Upsilon\Gamma$ be the natural (identity) transformation described in Section~\ref{sec.mor.from.double}. That the composite $\kappa\Gamma\circ\Gamma\lambda$ is the identity natural transformation $\Gamma \to \Gamma$ is immediate. To compute the other composite $\Upsilon\kappa \circ \lambda\Upsilon : \Upsilon \to \Upsilon$, take a double adjunction $\tt=(\Tt,(L_A,R_A),(L_B,R_B))$, so that
$$
\Upsilon\tt = (R_AL_A, R_BL_B, R_AL_B,R_BL_A, R_A\eps^BL_A, R_B\eps^AL_B).
$$
Then $\kappa \tt = K: \Tt\to \EMM{\Upsilon\tt}$ is the comparison functor, and hence
$$
\Upsilon\kappa\tt = \Upsilon(K) = (\phi^1,\phi^2,\phi^3,\phi^4),
$$
where 
$$\phi^1 = U_A\zeta^A KL_A\circ U_AG_A\eta^A, \qquad \phi^2 = U_B\zeta^B KL_B\circ U_BG_B\eta^B, 
$$
$$
\phi^3 = U_A\zeta^B KL_B\circ U_AG_B\eta^B,  \qquad \phi^4 = U_B\zeta^A KL_A\circ U_BG_A\eta^A.
$$
Here $(G_A,U_A)$, $(G_B, U_B)$ are adjoint pairs  given by 
$$
(\EMM{\Upsilon\tt}, (G_A,U_A), (G_B, U_B)):= \Gamma(R_AL_A, R_BL_B, R_AL_B,R_BL_A, R_A\eps^BL_A, R_B\eps^AL_B),
$$
 so
$$
G_A\eta^A = (R_AL_A\eta^A, R_BL_A\eta^A), \qquad G_B\eta^B = (R_BL_B\eta^B, R_AL_B\eta^B).
$$
Furthermore, using the definition of the comparison functor (applied to $L_AX$ and $L_BY$, for any objects $X\in \Xx$, $Y\in \Yy$), we obtain
$$
\zeta^AKL_A = (R_A\eps^AL_A, R_B\eps^BL_A), \qquad \zeta^BKL_B  = (R_B\eps^BL_B, R_A\eps^AL_B).
$$
Therefore,
$$
\phi^1 = R_A\eps^AL_A\circ R_AL_A\eta^A = R_AL_A, \qquad \phi^4 = R_B\eps^AL_A\circ R_BL_A\eta^A = R_BL_A,
$$
since $\eta^A$ is the unit and $\eps^A$ is the counit of the adjunction $(L_A,R_A)$. Similarly, $\phi^2 = R_BL_B$ and $\phi^3 = R_AL_B$. Thus $\Upsilon\kappa$ is the identity natural transformation, and since also $\lambda\Upsilon$ is the identity, so is their composite $\Upsilon\kappa \circ \lambda\Upsilon$. This proves that $\lambda$ is a unit and $\kappa$ is a counit of the adjunction $(\Gamma, \Upsilon)$. Since the unit $\lambda$ is a natural isomorphism,  $\Gamma$ is a full and faithful functor.
\end{proof}

\begin{corollary}\label{cor.equiv}
$(\Gamma, \Upsilon)$ is a pair of inverse equivalences if and only if, for all double adjunctions $\tt=(\Tt,(L_A,R_A),(L_B,R_B))\in \Adj(\Xx,\Yy)$, $(R_A,R_B)$ is a moritable pair.
\end{corollary}
\begin{proof}
The moritability of each of $(R_A,R_B)$ is paramount to the comparison functor $K$ being an equivalence, for all $\tt \in \Adj(\Xx,\Yy)$, i.e.\ to the natural transformation $\kappa$ in the proof of Proposition~\ref{prop.adjoint} being an isomorphism. Since the latter is the counit of adjunction $(\Gamma, \Upsilon)$, the corollary is an immediate consequence of  Proposition~\ref{prop.adjoint}.
\end{proof}

\subsection {Moritability}
The aim of this section is to determine, when a pair of functors is moritable in the sense of Definition~\ref{def.moritable}.
We begin with the following simple
\begin{lemma}\label{lem.iso}
Let $\TT$ be an object in $\Mor(\Xx,\Yy)$ and let $(f,g)$ be a morphism in $\EMM\TT$. Then $(f,g)$ is an isomorphism in $\EMM\TT$ if and only if  $f$ is an isomorphism in $\Xx$ and $g$ is an isomorphism in $\Yy$.
\end{lemma}
\begin{proof} If $(f,g)$ is an isomorphism in $\EMM\TT$, then clearly $f$ and $g$ are isomorphisms (as all functors, in particular forgetful functors, preserve isomorphisms). Conversely, let $f^{-1}$ be the inverse (in $\Xx$) of $f$ and $g^{-1}$ be the inverse of $g$ (in $\Yy$). By applying $f^{-1}$, $g^{-1}$ to both sides of equalities described by diagrams \eqref{fglin} and \eqref{fgcomp} one immediately obtains that $(f^{-1}, g^{-1})$ is a morphism in $\EMM\TT$.
\end{proof}

From now on we fix a double adjunction $\tt=(\Tt,(L_A,R_A),(L_B,R_B))$ on $\Xx$ and $\Yy$ (with counits $\eps^A$, $\eps^B$ and units $\eta^A$, $\eta^B$),  and set 
$$
\TT :=  \Upsilon(\tt) = (R_AL_A, R_BL_B, R_AL_B,R_BL_A, R_A\eps^BL_A, R_B\eps^AL_B)
$$
to be the corresponding Morita context. $K :\Tt\to \EMM\TT$ is the comparison functor. The aim of this section is to determine, when $K$ is an equivalence of categories.

\begin{proposition} \label{prop.ref}
Suppose that $(R_A, R_B)$ is a moritable pair and let $f$ be a morphism in $\Tt$. If both $R_Af$ and $R_Bf$ are isomorphisms, then so is $f$.
\end{proposition}
\begin{proof}
Note that $R_A = U_AK$ and $R_B = U_BK$, where $U_A: \EMM\TT\to \Xx$, $U_B: \EMM\TT\to \Yy$ are forgetful functors. If $R_Af$ and $R_Bf$ are isomorphisms, then, by Lemma~\ref{lem.iso} also $Kf$ is an isomorphism. Since an equivalence of categories reflects isomorphisms, also $f$ is an isomorphism.
\end{proof}

\begin{definition}\label{def.ref}
The pair $(R_A, R_B)$ is said to {\em reflect isomorphisms} if the fact that both $R_Af$ and $R_Bf$ are isomorphisms for a morphism $f\in \Tt$ implies that $f$ is an isomorphism. 
That is, the pair $(R_A,R_B)$ reflects isomorphisms if and only if the induced functor into the product category $\langle R_A,R_B\rangle : \Tt \to \Xx\times  \Yy$, $Z\mapsto (R_AZ, R_BZ)$, reflects isomorphisms.
\end{definition}

Note that if $R_A$ or $R_B$ reflects isomorphisms, then the pair $(R_A,R_B)$ reflects isomorphisms, but not the other way round. 
By Proposition \ref{prop.ref}, a moritable pair $(R_A,R_B)$ reflects isomorphisms.

To analyse the comparison functor $K$ further we assume the existence of particular colimits in $\Tt$. For any object $\X=((X,\rho^X),(Y,\rho^Y),v,w)\in\EMM{\TT}$ consider the following diagram
\begin{equation}\label{diag.col}
\xymatrix{L_AR_AL_A X\ar@<-.6ex>[dd]_{L_A\rho^X} \ar@<.6ex>[dd]^{\eps^AL_AX} && L_BR_BL_A X\ar[lldd]_{\eps^BL_AX} \ar[rrrdd]^{L_Bw} & L_AR_AL_BY \ar[rrdd]^{\eps^AL_BY} \ar[llldd]_{L_Av} && L_BR_BL_BY  \ar@<-.6ex>[dd]_{\eps^BL_BY} \ar@<.6ex>[dd]^{L_B\rho^Y}\\
&&&&&\\
L_AX &&&&& L_BY.}
\end{equation} 
From this point until Theorem~\ref{thm.Beck} assume that $\Tt$ has colimits of all such diagrams, and let 
$$
(D\X,\; d^A\X: L_AX\to D\X,\; d^B\X : L_BY\to D\X),
$$
be the colimit of  \eqref{diag.col}. Any morphism $(f,g) : \X \to \X'$ in $\EMM{\TT}$ determines a morphism of diagrams \eqref{diag.col} (i.e.\ it induces a natural transformation between functors from a six-object category to $\Tt$ that define diagrams \eqref{diag.col}) by applying suitable combinations of the $L$ and $R$ to $f$ and $g$. Therefore, by the universality of colimits,  there is a unique morphism $D(f,g): D\X\to D\X'$ in $\Tt$ which  satisfies the following identies
\begin{equation}\label{eq.D(f,g)}
d^A \X'\circ L_Af = D(f,g)\circ d^A\X, \qquad d^B \X'\circ L_Bf = D(f,g)\circ d^B\X.
\end{equation}
This construction yields a functor
$$
D: \EMM\TT \to \Tt, \qquad \X\mapsto D\X, \qquad (f,g)\mapsto D(f,g).
$$
\begin{proposition}\label{prop.adjoint.comp}
The functor $D$ is the left adjoint of the comparison functor $K$.
\end{proposition}
\begin{proof}
For any object $Z$ in $\Tt$ there is a cocone

{\small 
$$
~\hspace{-1cm} \xymatrix{L_AR_AL_A R_AZ\ar@<-.6ex>[dd]_{L_AR_A\eps^AZ} \ar@<.6ex>[dd]^{\eps^AL_AR_AZ} && L_BR_BL_A R_AZ\ar[lldd]_{\eps^BL_AR_AZ} \ar[rrrrdd]^{L_BR_B\eps^AZ} && L_AR_AL_BR_BZ \ar[rrdd]^{\eps^AL_BR_BZ} \ar[lllldd]_{L_AR_A\eps^BZ} && L_BR_BL_BR_BZ  \ar@<-.6ex>[dd]_{\eps^BL_BR_BZ} \ar@<.6ex>[dd]^{L_BR_B\eps^BZ}\\
&&&&&&\\
L_AR_AZ \ar[rrrd]^{\eps^AZ} &&&&&& L_BR_BZ  \ar[llld]_{\eps^BZ} \\
&&& Z.&&&}
$$}
This is a cocone under the diagram of type \eqref{diag.col} corresponding to the object $KZ$. By the universal property of colimits there is a unique morphism
$
\eps Z : DKZ\to Z,
$
such that
\begin{equation}\label{coun.eq}
\eps Z\circ d^AKZ = \eps^A Z, \qquad \eps Z\circ d^BKZ = \eps^B Z.
\end{equation}
This construction defines a natural transformation $\eps$ from the functor $DK$ to the identity functor on $\Tt$.  

For all objects $\X$ in $\EMM\TT$, define  $\eta^A\X: X\to R_AD\X$ and $\eta^B\X: Y\to R_BD\X$ as composites
$$
\eta^A\X = R_Ad^A\X\circ \eta^AX, \qquad \eta^B\X = R_Bd^B\X\circ \eta^BY.
$$
Using the naturality of  $\eps^A$, $\eps^B$, $\eta^A$ and $\eta^B$, the triangular identities  for units and counits of adjunctions, and the definition of $(D\X, d^A\X, d^B\X)$ as a cocone under the diagram \eqref{diag.col}, one can verify that the pair $\eta\X := (\eta^A\X, \eta^B\X)$ is a morphism in $\EMM\TT$. The assignment $\X \mapsto \eta\X$ defines a natural transformation $\eta$ from the identity functor on $\EMM\TT$ to $KD$. We now prove the triangular identities for $\eps$ and $\eta$.

Take any object $Z$ in $\Tt$ and compute
$$
R_A\eps Z\circ \eta^A KZ = R_A\eps Z\circ R_Ad^AKZ\circ\eta^A R_AZ = R_A\eps^AZ\circ \eta^A R_AZ =  R_AZ,
$$
where the second equality follows by \eqref{coun.eq} and the third by the triangular identities for $\eps^A$ and $\eta^A$. Similarly, $R_B\eps Z\circ \eta^B KZ = R_B Z$. This proves that the composite $K\eps\circ \eta K$ is the identity natural transformation on $K$. 

Next take any object $\X$ in $\EMM\TT$. Since $D\eta\X = D(\eta^A\X, \eta^B\X)$, $D\eta\X$ satisfies equalities \eqref{eq.D(f,g)}. In particular
$$
D\eta\X\circ d^A\X = d^AKD\X\circ L_AR_Ad^A\X\circ L_A\eta^AX.
$$
Therefore,
\begin{eqnarray*}
\eps D\X\circ D\eta\X\circ d^A\X &=& \eps D\X\circ d^AKD\X\circ L_AR_Ad^A\X\circ L_A\eta^AX\\
&=& \eps^AD\X \circ  L_AR_Ad^A\X\circ L_A\eta^AX \\
&=& d^A\X\circ \eps^AL_AX\circ L_A\eta^A X = d^A\X,
\end{eqnarray*}
where the second equality follows by \eqref{coun.eq}, the third one by the naturality of $\eps^A$, and the final one is one of the triangular identities for $\eps^A$ and $\eta^A$. Similarly,
$$
\eps D\X\circ D\eta\X\circ d^B\X  = d^B\X.
$$
The universality of colimits now yields $\eps D\X\circ D\eta\X = D\X$, i.e.\ the second triangular identity for $\eps$ and $\eta$.
\end{proof}

\begin{definition}\label{def.convert}
The pair $(R_A,R_B)$ is said to {\em convert colimits into coequalisers} if, for all objects $\X=((X,\rho^X),(Y,\rho^Y),v,w)\in\EMM{\TT}$, the diagrams
$$
\xymatrix{R_AL_AR_AL_A X\ar@<-.6ex>[rr]_-{R_AL_A\rho^X} \ar@<.6ex>[rr]^-{R_A\eps^AL_AX} && R_AL_A X \ar[rr]^-{R_Ad^A\X} && R_AD\X 
}
$$
and
$$
\xymatrix{R_BL_BR_BL_B Y\ar@<-.6ex>[rr]_-{R_BL_B\rho^Y} \ar@<.6ex>[rr]^-{R_B\eps^BL_BY} && R_BL_B Y \ar[rr]^-{R_Bd^B\X} && R_BD\X 
}
$$
are coequalisers in $\Xx$ and $\Yy$ respectively.
\end{definition}
\begin{lemma}\label{lem.f.f}
The functor $D$ is fully faithful if and only if the pair $(R_A,R_B)$ converts colimits into coequalisers.
\end{lemma}
\begin{proof}
For all objects $\X=((X,\rho^X),(Y,\rho^Y),v,w)\in\EMM{\TT}$, consider  the following diagram
$$
\xymatrix{R_AL_AR_AL_A X\ar@<-.6ex>[rr]_-{R_AL_A\rho^X} \ar@<.6ex>[rr]^-{R_A\eps^AL_AX} && R_AL_A X \ar[rd]_-{R_Ad^A\X} \ar[rr]^-{\rho^X} && X  \\
&&& R_AD\X .&
}
$$
Since the multiplication in  monad $R_AL_A$ is given by $R_A\eps^AL_A$, the top row is a coequaliser; see \cite[Proposition~4, Section~3.3]{BarWel:ttt}. Thus there exists a morphism $X\to R_AD\X$, which, by its uniqueness, must coincide with $\eta^A\X$. By considering a similar diagram for the other adjoint pair, one fits into it the other component of the unit of adjunction, $\eta^B\X$. If $(R_A,R_B)$ converts colimits into coequalisers, then (by the uniqueness of coequalisers) both $\eta^A\X$ and $\eta^B\X$ are isomorphisms. Hence, by Lemma~\ref{lem.iso}, $\eta$ is an isomorphism, i.e.\ $D$ is fully faithful. Conversely, if $\eta$ is an isomorphism, then both $(R_AD\X, R_Ad^A\X)$ and $(R_BD\X, R_Bd^B\X)$ are (isomorphic to) coequalisers, i.e.\ the pair $(R_A,R_B)$ converts colimits into coequalisers.
\end{proof}

The main result of this section is contained in the following precise moritability theorem (Beck's theorem for double adjunctions). For this theorem the existence of colimits of diagrams \eqref{diag.col} needs not to be assumed {\em a priori}. 
\begin{theorem}\label{thm.Beck}
Let $\tt=(\Tt,(L_A,R_A),(L_B,R_B))$ be a double adjunction.  Then the pair $(R_A,R_B)$ is moritable if and only if $\Tt$ has colimits of all the diagrams \eqref{diag.col} and the pair $(R_A,R_B)$ reflects isomorphisms and converts colimits into coequalisers.
\end{theorem}
\begin{proof}
Assume that $K$ is an equivalence of categories. Consider a diagram of the form \eqref{diag.col} in $\Tt$. We can choose $\X=KZ$ for some $Z\in\Tt$ and we know that $(Z,\varepsilon^A,\varepsilon^B)$ is a cocone for this diagram.
Now apply the functor $K$ to diagram \eqref{diag.col}, then we claim that $(KZ,K\varepsilon^A,K\varepsilon^B)$ is a colimit for the new diagram in $\EMM \TT$. To check this, consider any cocone $(\H,(f^A,g^A),(f^B,g^B))$ on the diagram. Define $(h,k):KZ\to \H$ by putting $h=f^A\circ \eta^AR_AZ$ and $k=g^B\circ \eta^BR_BZ$. One can verify that $(h,k)$ is indeed a morphism in $\EMM \TT$ (apply the fact that $\H$ is a cocone and that $(f^A,g^A)$ and $(f^B,g^B)$ are morphisms in $\EMM\TT$, together with the adjunction properties of $(L_A,R_A)$ and $(L_B,R_B)$). Since $K$ is an equivalence of categories, it reflects colimits and therefore  $(Z,\varepsilon^A,\varepsilon^B)$ is the colimit of the original diagram \eqref{diag.col} in $\Tt$. Furthermore, $(R_A,R_B)$ reflects isomorphisms by Proposition~\ref{prop.ref} and it converts colimits into coequalisers by Lemma~\ref{lem.f.f}. 

In the converse direction, the counit $\eta$ is an isomorphism by Lemma~\ref{lem.f.f}. Applying $R_A$ to the first column of the cocone defining $\eps Z$ we obtain the following diagram
$$
\xymatrix{R_AL_AR_AL_A R_AZ\ar@<-.6ex>[d]_{R_AL_AR_A\eps^AZ} \ar@<.6ex>[d]^{R_A\eps^AL_AR_AZ} && R_AL_AR_AL_A R_AZ\ar@<-.6ex>[d]_{R_AL_AR_A\eps^AZ} \ar@<.6ex>[d]^{R_A\eps^AL_AR_AZ}\\
R_AL_AR_AZ \ar[d]_{R_A\eps^AZ} &&R_AL_AR_AZ \ar[d]^{R_Ad^AKZ} \\
R_AZ&&\ar[ll]_-{R_A\eps Z}R_ADKZ .}
$$
The first column is a (contractible) coequaliser, the second is a coequaliser by the assumption that $(R_A,R_B)$ converts colimits into coequalisers. Thus $R_A\eps Z$ is an isomorphism. Similarly, $R_B\eps Z$ is an isomorphism. Since the pair $(R_A,R_B)$ reflects isomorphisms, also $\eps Z$ is an isomorphism in $\Tt$.
\end{proof}

\section{Morita theory}\label{sec.Moritath}

In this section we study when, given a Morita context $\TT=(\A,\B,\T,\bhT,\ev,\hev)$ on $\Xx$ and $\Yy$, the bialgebra functors $\T$ and $\bhT$ induce an equivalence of categories between the categories $\EM\Xx\A$ and $\EM\Yy\B$. In particular, we prove that if $\Xx$ and $\Yy$ have coequalisers and $U_A$ and $U_B$ preserve coequalisers, then \emph{any} equivalence between two categories $\EM\Xx\A$ and $\EM\Yy\B$ of algebras of monads $\A$ and $\B$ is induced by a Morita context in $\Mor(\Xx,\Yy)$.

\subsection{Preservation of coequalisers by algebras}
Consider a monad $\A=(A,\m^A,\biota^A)$ on a category $\Xx$ and a functor $S:\Xx\to \Yy$. A pair $(S,\sigma)$ is called a {\em right $\A$-algebra functor} if $(S,S,\sigma)$ is a $\bf{\Yy}\hbox{-}\A$ bialgebra functor, where $\bf{\Yy}$ is the trivial monad on $\Yy$. Thus $(S,\sigma)$ is a right $\A$-algebra functor if and only if $\sigma:SA\to S$ is a natural transformation such that $\sigma\circ S\m^A = \sigma\circ \sigma A$ and $S=\sigma\circ S\biota^A$. Similarly, one introduces the notion of a left $\A$-algebra functor. Note that if $\T=(T,\lambda,\rho)$ is an $\A$-$\B$ bialgebra functor, then $(T,\lambda)$ is a left $\A$-algebra functor and $(T,\rho)$ is a right $\B$-algebra functor.

\begin{lemma}\label{lem.SA}
Let $\A=(A,\m^A,\biota^A)$ be a monad on $\Xx$ and $\S=(S,\sigma)$ a right $\A$-algebra functor. Any coequaliser  preserved by $SA$ is also preserved by $S$. If $(S',\sigma')$ is a left $\A$-algebra functor, then any coequaliser preserved by $AS'$ is also preserved by $S'$.
\end{lemma}

\begin{proof}
Consider a coequaliser 
$$
\xymatrix{
X \ar@<.5ex>[rr]^f \ar@<-.5ex>[rr]_g && Y \ar[rr]^z && Z
}
$$
in $\Xx$ and assume that it is preserved by $SA$. 
Applying the functors $S$ and $SA$ to this coequaliser, one obtains the following diagram in $\Yy$
\[
\xymatrix{
SX \ar@<-.5ex>[d]_{S \biota^A X} \ar@<.5ex>[rr]^-{S f} \ar@<-.5ex>[rr]_-{S g} && SY \ar@<-.5ex>[d]_{S\biota^A Y} \ar[rr]^-{S z} && SZ \ar@<-.5ex>[d]_{S \biota^A Z}  \\
SA X \ar@<-.5ex>[u]_{\sigma X} \ar@<.5ex>[rr]^-{SA f} \ar@<-.5ex>[rr]_-{SA g} && SA Y \ar@<-.5ex>[u]_{\sigma Y} \ar[rr]^-{SA z} && SAZ\ . \ar@<-.5ex>[u]_{\sigma Z}
}
\]
By assumption the lower row is a coequaliser.
Suppose that there exists a pair $(h,H)$, where $H$ is an object in $\Yy$ and $h:SY\to H$ is a morphism in $\Yy$ such that $h\circ Sf=h\circ Sg$. 
Since $\sigma$ is a natural transformation,  $h\circ \sigma Y\circ SAf = h\circ \sigma Y\circ SAg$. By the univeral property of the coequaliser $(SAZ, SAz)$, there is a unique morphism $k':SAZ\to H$ such that $k'\circ SAz= h\circ \sigma Y$. This, together with the naturality of $\biota^A$ and the unitality  of the right $\A$-algebra $\S$ imply  that $k'\circ S\biota^AZ\circ Sz= h$. Then $k=k'\circ S\biota^AZ: SZ\to H$ is a unique morphsim such that $h=k\circ Sz$.

The second statement is verified by a similar computation. 
\end{proof}

\begin{lemma}\label{lem.coeqA^A}
Let $f,g:\X\to \Y$ be morphisms in $\Xx^\A$. Suppose that the coequaliser $(E,e)$ of $(U_A(f),U_A(g))$ exists in $\Xx$. If $AA$ preserves this coequaliser, then the coequaliser $(\E,\epsilon)$ of the pair $(f,g)$ exists in $\Xx^A$ and $U_A(\E,\epsilon)=(E,e)$.
\end{lemma}

\begin{proof}
By Lemma~\ref{lem.SA}, $A$ preserves the coequaliser $(E,e)$. The universal property of  coequalisers implies the existence of a unique map $\rho^E:AE\to E$ such that $\rho^E\circ Ae=e\circ \rho^Y$ and $\rho^E\circ \biota^AE=E$. Using the fact that the coequaliser $(E,e)$ is preserved by $AA$, one  checks that $\rho^E$ defines a (associative) left action of $A$ on $E$, i.e.\ $(E,\rho^E)$ is an object of $\EM\Xx\A$.
\end{proof}

The following standard lemma relates various preservation properties for coequalisers. We include a brief proof for completeness.

\begin{lemma}\label{coequalizersEM}
If $\Xx$ has (all) coequalisers, then the following statements are equivalent:
\begin{enumerate}[(i)]
\item $A:\Xx\to\Xx$ preserves coequalisers;
\item $AA:\Xx\to\Xx$ preserves coequalisers;
\item $\EM\Xx\A$ has (all) coequalisers and they are preserved by $U_A:\EM\Xx\A\to \Xx$;
\item $U_A:\EM\Xx\A\to \Xx$ preserves coequalisers.
\end{enumerate}
\end{lemma}

\begin{proof}
Implications $\ul{(i)\Rightarrow(ii)}$ and $\ul{(iii)\Rightarrow(iv)}$ are obvious. 

$\ul{(ii)\Rightarrow(iii)}$. Follows by Lemma~\ref{lem.coeqA^A}.

$\ul{(iv)\Rightarrow(i)}$. The free algebra functor $\fr A:\Xx\to\EM\Xx\A$ preserves coequalisers, since it is a left adjoint functor. By assumption  $U_A$ preserves coequalisers as well and $A=U_A\fr A$, so the statement $(i)$ follows.
\end{proof}

\subsection{Morita contexts and equivalences of categories of algebras}
Let $\B= (B,\m^B,\biota^B)$ be a monad on $\Yy$ and let $(T:\Yy\to\Xx,\rho)$ be a right $\B$-algebra functor.
For any $(Y,\rho^Y)$ in $\EM\Yy\B$, let $(T_BY,\tau Y)$ be the following coequaliser in $\Xx$ (if it exists)
\begin{equation}\label{eq.TBX}
\xymatrix{
TBY \ar@<.5ex>[rr]^-{\rho Y} \ar@<-.5ex>[rr]_-{T\rho^Y} && TY \ar[rr]^-{\tau Y} && T_BY.
}
\end{equation}
Similarly, given a monad $\A=(A,\m^A,\biota^A)$ on $\Xx$ and  a right $\A$-algebra functor $(\hT:\Xx\to\Yy,\hrho)$, consider an object $(X,\rho^X)$ in $\EM\Xx\A$,  and set $(\hT_AX,\htau X)$ to be the following coequaliser in $\Yy$ (if it exists)
\begin{equation}\label{eq.TAY}
\xymatrix{
\hT AX \ar@<.5ex>[rr]^-{\hrho X} \ar@<-.5ex>[rr]_-{\hT\rho^X} && \hT X \ar[rr]^-{\hat \tau X} && \hT_AX .
}
\end{equation}
Finally, recall that for any $\B$-algebra $(Y,\rho^Y)$, the following diagram is a contractible coequaliser in $\Yy$ and a (usual) coequaliser in $\EM\Yy\B$,
\begin{equation}\label{trivial}
\xymatrix{
BBY \ar@<.5ex>[rr]^-{\m^B Y} \ar@<-.5ex>[rr]_-{B\rho^Y} && BY \ar[rr]^-{\rho^Y} && Y.
}
\end{equation}
As in Section~\ref{sec.eil-moo}, the free--forgetful adjunctions for $\A$ and $\B$ are denoted by $(\fr A,U_A)$, $(\fr B, U_B)$, respectively. The counits are denoted by  $\feps^A$, $\feps^B$. 
 We are now ready to state the following lifting theorem.

\begin{proposition}\label{prop.TB}
Let $\A$ be a monad on $\Xx$ and $\B$ a monad on $\Yy$. There is a bijective correspondence between the following data:
\begin{enumerate}[(i)]
\item $\A$-$\B$ bialgebra functors $\T$, such that coequalisers of the form \eqref{eq.TBX} exist in $\Xx$ for any $Y\in\Yy$, and they are preserved by $AA$;

\item functors $T_B: \EM\Yy\B\to \EM\Xx\A$ such that $AAU_AT_B:\EM\Yy\B\to\Xx$ preserves coequalisers of the form \eqref{trivial} for all $(Y,\rho^Y)\in \EM\Yy\B$.
\end{enumerate}
Let $\T$ be a bialgebra functor as in $(i)$ and $T_B$ the corresponding functor of $(ii)$, then given a functor $P:\Xx\to\Ww$, the functor $P$ preserves coequalisers of the form \eqref{eq.TBX} if and only if the functor $PU_AT_B:\EM\Yy\B\to\Ww$ preserves coequalisers of the form \eqref{trivial}.
\end{proposition}

\begin{proof}
$\ul{(i)\Rightarrow(ii)}$.
Given a $\B$-algebra $(Y,\rho^Y)$, define $T_BY$ by \eqref{eq.TBX}. Then it follows by Lemma~\ref{lem.coeqA^A} that there is a morphism $\rho^{T_BY}: AT_BY\to T_B Y$ such that $(T_BY, \rho^{T_BY})$ is an object in $\EM\Xx\A$.
For a morphism $f:Y\to Y'$ in $\EM\Yy\B$, the universality of the coequaliser induces  a morphism $T_Bf:T_BY\to T_BY'$ such that $T_Bf\circ \tau Y=\tau Y'\circ Tf$. To check that $T_Bf$ is a morphism in $\EM\Xx\A$, one can proceed as follows. By Lemma \ref{lem.SA},  $A$ preserves coequalisers of the form \eqref{eq.TBX}, so in particular, $A\tau Y$ is an epimorphism. Therefore, it is enough to verify that $T_Bf\circ\rho^{T_BY}\circ A\tau Y= \rho^{T_BY'}\circ AT_Bf\circ A\tau Y$, which follows from the defining properties of $\rho^{T_BY}$, $\rho^{T_BY'}$ and $T_Bf$, combined with the naturality of the right action $\rho$ of $T$. Thus there is a well-defined functor $T_B:\EM\Yy\B\to\EM\Xx\A$.
Consider now a functor $P:\Xx\to\Ww$ which preserves all the coequalisers of the form \eqref{eq.TBX}.
For any $(Y,\rho^Y)\in\EM\Yy\B$, construct the following diagram in $\Xx$
\[
\xymatrix{
PTBBBY \ar@<.5ex>[rr]^-{PTB\m^BY} \ar@<-.5ex>[rr]_-{PTBB\rho^Y} \ar@<-.5ex>[d]_-{PT\m^B BY} \ar@<.5ex>[d]^-{P\rho BBY} 
&& PTBBY \ar[rr]^-{PTB\rho^Y} \ar@<-.5ex>[d]_-{PT\m^BY} \ar@<.5ex>[d]^-{P\rho BY}
&& PTBY  \ar@<-.5ex>[d]_-{PT \rho^Y} \ar@<.5ex>[d]^-{P\rho Y} \\
PTBBY \ar@<.5ex>[rr]^-{PT\m^BY} \ar@<-.5ex>[rr]_-{PTB\rho^Y} \ar[d]_{P\tau{BBY} } 
&& PTBY \ar[rr]^-{PT\rho^Y} \ar[d]_{P\tau{BY}} && PTY \ar[d]^{P\tau Y} \\
PU_AT_B BBY \ar@<.5ex>[rr]^-{PT_B\m^BY} \ar@<-.5ex>[rr]_-{PT_BB\rho^Y} 
&& PU_AT_BBY \ar[rr]^-{PT_B\rho^Y} && PU_AT_BY \, .
}
\]
The first and second rows of this diagram are coequalisers, as they are obtained by applying functors to the {\em contractible} coequaliser \eqref{trivial}. All three columns are coequalisers, since they are coequalisers of type \eqref{eq.TBX} to which the functor $P$ is applied. The diagram chasing arguments then yield that the lower row is a coequaliser too. The first part of the proof is then completed by setting $P=AA$.

 $\ul{(ii)\Rightarrow(i)}$. 
Define $T=U_AT_B\fr B$, and natural transformations $\rho=U_AT_B\feps^B\fr B = U_AT_B\m^B:TB\to T$ and $\lambda = U_A\feps^A T_B F_B:AT\to T$. Then $(T,\lambda,\rho)$ is an $\A$-$\B$ bialgebra functor.  
Note that $U_AT_B$ is also a left $\A$-algebra functor with action $U_A\feps^AT_B$.
Lemma \ref{lem.SA} implies that $U_AT_B$ preserves the coequaliser \eqref{trivial}. This means that, for any $\Y= (Y,\rho^Y)\in\EM\Yy\B$, the following diagram is a coequaliser in $\Xx$
\begin{equation}\label{eq.TBXalt}
\xymatrix{
U_AT_B\fr BBY \ar@<.5ex>[rr]^-{U_AT_B\m^B Y} \ar@<-.5ex>[rr]_-{U_AT_B\fr B\rho^Y} && U_AT_B\fr BY \ar[rr]^-{U_AT_B\rho^Y} && U_AT_B\Y \, .
}
\end{equation}
This diagram is exactly the coequaliser \eqref{eq.TBX} in this situation. Applying the functor $AA$ (resp.\ any functor $P:\Xx\to\Ww$) to the coequaliser \eqref{eq.TBXalt} yields the same result as applying the functor $AAU_AT_B$ (resp.\ $PU_AT_B$) to \eqref{trivial}. Therefore, the functor $AA$ (resp. $P$) preserves coequalisers \eqref{eq.TBX}.
\end{proof}

In the following  the notation introduced in Proposition~\ref{prop.TB} is used. For an $\A$-$\B$ bialgebra functor $\T$, $T_B$ denotes the functor defined by coequalisers \eqref{eq.TBX},  and $T=U_AT_B\fr B$.

\begin{lemma}\label{lem.TA}
Suppose that the coequalisers of the form \eqref{eq.TBX} and \eqref{eq.TAY} exist and they are preserved by $AA$, $\hT A$ and by $BB$ respectively, then the functor $\hT_A:\EM\Xx\A\to\EM\Yy\B$ preserves coequalisers of the form \eqref{eq.TBX}.
\end{lemma}

\begin{proof}
Note that the functor $\hT_A$ is well-defined by (the dual version of) Proposition~\ref{prop.TB}. We can now consider the following diagram in $\Yy$:
\begin{equation}\label{diag.coeqTATB}
\xymatrix{
\hT ATBY \ar@<.5ex>[rr]^-{\hT A\rho Y} \ar@<-.5ex>[rr]_-{\hT AT\rho^Y} \ar@<-.5ex>[d]_-{\hrho TBY} \ar@<.5ex>[d]^-{\hT \lambda BY} && \hT ATY \ar[rr]^-{\hT A\tau Y} \ar@<-.5ex>[d]_-{\hrho TY} \ar@<.5ex>[d]^-{\hT \lambda Y} && \hT AT_BY \ar@<-.5ex>[d]_-{\hrho T_BY} \ar@<.5ex>[d]^-{\hT \rho^{T_BY}} \\
\hT TBY \ar@<.5ex>[rr]^-{\hT \rho Y} \ar@<-.5ex>[rr]_-{\hT T\rho^Y} \ar[d]^{\htau {TBY}} && \hT TY \ar[rr]^-{\hT\tau Y} \ar[d]^{\htau {TY}} && \hT T_BY \ar[d]^{\htau {T_BY}}\\
\hT_ATBY \ar@<.5ex>[rr]^-{\hT_A \rho Y} \ar@<-.5ex>[rr]_-{\hT_A \rho^Y} && \hT_A TY \ar[rr]^-{\hT_A\tau Y} && \hT_AT_BY \, .
}
\end{equation}
By assumption, the first row is a coequaliser; by Lemma \ref{lem.SA}, the second row is a coequaliser too. All three columns are coequalisers by definition.
Therefore, the lower row is an equaliser as well. If we apply the functor $BB$ to this diagram, the same reasoning yields that $BB$ preserves the equaliser in the lower row of  \eqref{diag.coeqTATB}, therefore it is an equaliser in $\EM\Yy\B$ by Lemma~\ref{lem.coeqA^A}.
\end{proof}

Recall that a {\em wide Morita context} $(F,G,\mu,\tau)$ between categories $\Cc$ and $\Dd$, consists of two functors $F:\Cc\to \Dd$ and $G:\Dd\to\Cc$, and two natural transformations $\mu:GF\to \Cc$ and $\tau:FG\to \Dd$, satisfying $F\mu=\tau F$ and $\mu G=G\tau$. 
In \cite{CIGT:wide},  {\em left} (resp.\ {\em right}) wide Morita contexts are studied. In this setting, $\Cc$ and $\Dd$ are abelian (or Grothendieck) categories and $F$ and $G$ are left (resp.\ right) exact functors. Left and right wide Morita contexts are used to characterise equivalences between the categories $\Cc$ and $\Dd$. 
In the remainder of this section we study wide Morita contexts between categories of algebras over monads. This allows us to weaken the assumptions made in \cite{CIGT:wide}. Moreover, we study the relationship between wide Morita contexts and Morita contexts.

Consider monads $\A$ on $\Xx$ and $\B$ on $\Yy$. A pair of functors $\hT_A:\EM\Xx\A\to\EM\Yy\B$ and $T_B:\EM\Yy\B\to\EM\Xx\A$ is said to be {\em algebraic} if the functors $AAU_AT_B$ and $U_B\hT_AA AU_AT_B$ preserve coequalisers of the form \eqref{trivial} for all $(Y,\rho^Y)\in\EM\Yy\B$, and the functors $BBU_B\hT_A$ and $U_AT_BB BU_B\hT_A$ preserve coequalisers of the form 
\begin{equation}\label{trivial2}
\xymatrix{
AAX \ar@<.5ex>[rr]^-{\m^A X} \ar@<-.5ex>[rr]_-{A\rho^X} && AX \ar[rr]^-{\rho^X} && X\, ,
}
\end{equation}
for all $(X,\rho^X)\in \EM\Xx\A$.
By an {\em algebraic} wide Morita context we mean a wide Morita context $(T_B,\hT_A,\homega,\omega)$ between categories of algebras $\EM\Yy\B$ and $\EM\Xx\A$, such that the functors $T_B$ and $\hT_A$ form an algebraic pair of functors.

\begin{remark}\label{rem.morita}
If the forgetful functors $U_B$ and $U_A$ preserve coequalisers and the categories $\EM\Yy\B$ and $\EM\Xx\A$ have all coequalisers (e.g.\  $\EM\Yy\B$ and $\EM\Xx\A$ are abelian categories, as in \cite{CIGT:wide}) or equivalently the categories $\Yy$ and $\Xx$ have all coequalisers, see Lemma \ref{coequalizersEM}, then the situation simplifies significantly. In this case, for any bialgebra functors $\T$ and $\bhT$, all coequalisers of the form \eqref{eq.TBX} and \eqref{eq.TAY} exist and, by Lemma \ref{coequalizersEM}, are preserved by $AA$ and $BB$ respectively. 
Furthermore, in this situation, $(T_B,\hT_A)$ is an algebraic pair if and only if the functors $T_B$ and $\hT_A$ preserve coequalisers (of the form \eqref{trivial} and \eqref{trivial2} respectively). In particular, it is an adjoint algebraic pair if $\hT_A$ preserves coequalisers (as any left adjoint preserves all coequalisers). Also, if $\EM\Yy\B$ and $\EM\Xx\A$ are abelian, then a right wide Morita context between  $\EM\Yy\B$ and $\EM\Xx\A$ is always algebraic.
\end{remark}

\begin{theorem}\label{th.wide}
Let $\A$ be a monad on $\Xx$ and $\B$ a monad on $\Yy$.
There is a bijective correspondence between the following data:
\begin{enumerate}[(i)]
\item Morita contexts  $\TT=(\A,\B,\T,\bhT,\ev,\hev)$ on $\Xx$ and $\Yy$, such that all coequalisers of the form \eqref{eq.TBX} and \eqref{eq.TAY} exist, and they are preserved by $AA$, $\hT A$ and $BB$, $TB$ respectively;
\item functors $T_B:\EM\Yy\B\to \EM\Xx\A$ and $\hT_A:\EM\Xx\A\to\EM\Yy\B$, and natural transformations $\homega:\hT_AT_B\to {\EM\Yy\B}$ and $\omega:T_B\hT_A\to {\EM\Xx\A}$ such that $(T_B,\hT_A,\homega,\omega)$ is an algebraic wide Morita context between $\EM\Yy\B$ and $\EM\Xx\A$.
\end{enumerate}
\end{theorem}

\begin{proof}
$\ul{(i)\Rightarrow(ii)}$.
The existence of the functors $\hT_A$ and $T_B$, forming an algebraic pair, follows from Proposition~\ref{prop.TB}. Moreover, in light of Lemma~\ref{lem.TA}, $\hT_A$ preserves coequalisers of the form \eqref{eq.TBX} and $T_B$ preserves coequalisers of the form \eqref{eq.TAY}.
Consider again diagram \eqref{diag.coeqTATB}
in which now all the rows and columns are coequalisers. 
The natural transformation $\hev$ induces a map $\rho^Y\circ \hev Y : \hT TY \to Y$ that equalises both the horizontal and vertical arrows. A diagram chasing argument affirms the existence of a map $\homega Y: \hT_A T_B Y\to Y$. The naturality of $\hev$ implies that $\homega$ is a natural transformation as well.

Similarly, one defines $\omega:T_B\hT_A\to {\EM\Xx\A}$. To check that the compatibility conditions between $\homega$ and $\omega$ hold, one starts with objects $TB\hT ATBY$ and $\hT ATB\hT AX$ and constructs coequaliser diagrams resulting in $T_B\hT_AT_BY$ and $\hT_AT_B\hT_AX$, respectively. The compatibility conditions between $\ev$ and $\hev$ (diagrams \eqref{morita}-\eqref{balanced}), as well as the facts that they are bialgebra morphisms and  that $T$ and $\hT$ are bialgebra functors, induce the needed compatibility conditions for $\homega$ and $\omega$. 

$\ul{(ii)\Rightarrow(i)}$.
By Proposition~\ref{prop.TB} there exist bialgebra functors $\T$ and $\bhT$, where $T=U_AT_B\fr B$ and $\hT=U_B\hT_A\fr A$, such that  all coequalisers of the form \eqref{eq.TBX} and \eqref{eq.TAY} exist, and they are preserved by $AA$, $\hT A$ and $BB$, $TB$ respectively. Define 
\begin{eqnarray*}
\xymatrix{
\ev: T\hT= U_AT_B\fr B U_B\hT_A\fr A \ar[rr]^-{U_AT_B\feps^B\hT_A\fr A} && U_AT_B\hT_A\fr A \ar[rr]^-{U_A\omega \fr A} && U_A\fr A =A\, , }\\
\xymatrix{
\hev: \hT T= U_B\hT_A\fr AU_AT_B\fr B \ar[rr]^-{U_B\hT_A\feps^AT_B\fr B} && U_B\hT_AT_B\fr B \ar[rr]^-{U_B\homega \fr B} && U_B\fr B=B \, ,
}
\end{eqnarray*}
where $\feps^B$ and $\feps^A$ are the counits of the adjunctions $(\fr B,U_B)$ and $(\fr A,U_A)$. 
Then $(\A,\B,\T,\bhT,\ev,\hev)$ is a Morita context.
\end{proof}

Consider a Morita context $(\A,\B,\T,\bhT,\ev,\hev)$ on $\Xx$ and $\Yy$. Suppose that the coequalisers of the form \eqref{eq.TBX} and \eqref{eq.TAY} exist and are preserved by  $AA$, $\hT A$ and $BB$, $TB$ respectively. 
Because $\ev$ and $\hev$ are respectively $B$ and $A$-balanced, the universal property of the coequaliser implies the existence of
natural transformations $\pi:T_B\hT\to A$ and $\hpi:\hT_AT\to B$ such that, for all objects $X\in\Xx$ and $Y\in\Yy$,
$$ \ev X = \pi X\circ \tau \hT X \quad {\rm and}\quad \hev Y = \hpi Y \circ \htau TY. $$
 Since the coequalisers \eqref{eq.TBX} and \eqref{eq.TAY} are preserved by $\hT_A$ and $T_B$, respectively (see Proposition~\ref{prop.TB}), for all objects $(X,\rho^X)\in \Xx^\A$, $(Y,\rho^Y)\in \Yy^\B$, 
\begin{equation}\label{piomega}
\rho^X\circ \pi X = \omega X \circ T_B \htau X \quad {\rm and} \quad \rho^Y\circ \hpi Y = \homega Y \circ \hT_A \tau Y ,
\end{equation}
where $\omega:T_B\hT_A\to \EM\Xx\A$ and $\homega:\hT_AT_B\to \EM\Yy\B$ are defined in Theorem~\ref{th.wide}.
With this notation, we have the following

\begin{theorem}\label{th.morita}
Let $\A$ be a monad on $\Xx$ and $\B$ a monad on $\Yy$.
\begin{enumerate}
\item There is a bijective correspondence between the following data:
\begin{enumerate}[(i)]
\item pairs of adjoint functors $(T_B:\EM\Yy\B\to\EM\Xx\A,\hT_A:\EM\Xx\A\to\EM\Yy\B)$ such that $(T_B,\hT_A)$ is an algebraic pair and $T_B$ is fully faithful;
\item algebraic wide Morita contexts $(T_B,\hT_A,\omega,\homega)$ such that $\homega$ is a natural isomorphism;
\item Morita contexts $(\A,\B,\T,\bhT,\ev,\hev)$ such that all coequalisers of the form \eqref{eq.TBX} and \eqref{eq.TAY} exist and are preserved by  $AA$, $\hT A$ and $BB$, $TB$, respectively, and there exists a natural transformation $\hchi:\Yy\to \hT_AT$ such that $\hpi \circ\hchi= \biota^B$.
\end{enumerate}
\item There is a bijective correspondence between the following data:
\begin{enumerate}[(i)]
\item inverse pairs of equivalences $(T_B:\EM\Yy\B\to\EM\Xx\A,\hT_A:\EM\Xx\A\to\EM\Yy\B)$ such that $(T_B,\hT_A)$ is an algebraic pair;
\item algebraic wide Morita contexts $(T_B,\hT_A,\omega,\homega)$, such that $\homega$ and $\omega$ are natural isomorphisms;
\item Morita contexts $(\A,\B,\T,\bhT,\ev,\hev)$ such that all coequalisers of the form \eqref{eq.TBX} and \eqref{eq.TAY} exist and are preserved by  $AA$, $\hT A$  and $BB$, $TB$ respectively, and there exist natural transformations $\hchi:\Yy\to \hT_AT$ and $\chi:\Xx\to T_B\hT$ such that $\hpi \circ\hchi = \biota^B$ and $\pi\circ\chi=\biota^A$. \end{enumerate}
\end{enumerate}
\end{theorem}

\begin{proof}
We only prove part (1), as (2) follows by combining (1) with its $\Xx$-$\Yy$ symmetric version. 

$\ul{(ii)\Leftrightarrow(i)}$. Let  $(T_B,\hT_A,\omega,\homega)$ be an algebraic wide Morita context and let $\hvarpi$ denote the inverse natural transformation of $\homega$. 
For any $\X\in\EM\Xx\A$, 
$$\hT_A\omega \X\circ \hvarpi \hT_A\X = \homega \hT_A\X\circ \hvarpi \hT_A\X = \hT_A\X.$$
The first equality is the compatibility condition between $\omega$ and $\homega$ in the wide Morita context (see Theorem \ref{th.wide}). Similarly, $T_B\homega\Y\circ \hvarpi T_B\Y=T_B\omega\Y\circ T_B\hvarpi\Y = T_B\Y$,  for all $\Y\in\EM\Yy\B$. Hence $(\hT_A,T_B)$ is  an adjoint pair with unit  $\hvarpi$ and counit $\omega$. 
$T_B$ is fully faithful since the unit of adjunction is a natural isomorphism. Conversely, if $(\hT_A,T_B)$ is an algebraic adjoint pair and $T_B$ is fully faithful, then similar computation confirms that $(T_B,\hT_A,\omega,\homega)$ is an algebraic wide Morita context, where $\omega$ is the counit and $\homega$ is the inverse of the unit of the adjunction.

$\ul{(ii)\Leftrightarrow(iii)}$. 
Assume first that the statement $(iii)$ holds.
In light of Theorem \ref{th.wide} suffices it to construct the natural inverse $\hvarpi:\EM\Yy\B\to\hT_AT_B$ of $\homega$. 
For any algebra $\Y=(Y,\rho^Y)\in\EM\Yy\B$, set $\hvarpi Y= \hT_A \tau Y \circ\hchi Y$.
The 
naturality of $\hvarpi$ follows by the preservation of coequalisers by $\hT_A$ and $T_B$ (see  diagram \eqref{diag.coeqTATB}). Take any $(Y,\rho^Y)\in\EM\Yy\B$ and compute
$$
\homega Y\circ \hvarpi Y = \homega Y\circ \hT_A \tau Y \circ\hchi Y = \rho^Y\circ\hpi Y\circ \hchi Y = \rho^Y \circ\biota^BY = Y.
$$
The second equality follows by \eqref{piomega}.
Furthermore, for any $(Y,\rho^Y)\in\EM\Yy\B$, 
\begin{eqnarray*}
\hvarpi Y\circ \homega Y\circ\hT_A\tau Y \circ\htau {TY} 
&=&  \hT_A\tau Y\circ \hchi Y \circ \rho^Y\circ \hpi Y \circ \htau {TY}\\
&=& \hT_A\tau Y\circ \hchi Y \circ \rho^Y\circ \hev Y\\
&=&  \hT_A\tau Y\circ \hT_A T\rho^Y \circ \hchi BY \circ \hev Y\\
&=& \hT_A\tau Y\circ \hT_A \rho Y \circ \hT_AT\hev Y\circ \hchi \hT TY \\
&=& \hT_A\tau Y\circ \hT_A \lambda Y\circ \hT_A  \ev T Y \circ \hchi \hT TY.
\end{eqnarray*}
The first two equalities follow by the definitions of $\hvarpi$ and $\hpi$, and by\eqref{piomega}.
The third step is a consequence of the naturality of $\hchi$. The equalising property of the maps $\hT\tau Y$ and $\htau {TY}$, as well as the equivalent expressions for the Godement product  are used in the fourth step. The final equality  follows by the first of diagrams \eqref{morita} that express compatibility between $\ev$ and $\hev$ in the Morita context. On the other hand,
\begin{eqnarray*}
\hT_A\lambda Y\circ \hT_A\ev TY \!\!\!\!\! & \circ & \!\!\!\!\! \htau {T\hT TY}
= \htau {TY} \circ \hT\lambda Y\circ \hT\ev TY 
= \htau {TY} \circ \hrho TY \circ \hT\ev TY \\
&=& \htau_{TY} \circ \hlambda TY \circ \hev\hT TY
= \htau {TY} \circ \hlambda TY \circ \hpi\hT TY \circ \htau {T\hT TY},
\end{eqnarray*}
where the first equality follows by the naturality of $\htau$, the second is the equalising property of $\htau TY$ (recall from Section~\ref{sec.morita} that $\rho^{TY} = \lambda Y$), the third equality follows by the second of diagrams \eqref{morita}. The final equality is the defining property of $\hpi Y$.  Since $\htau {T\hT TY}$ is an epimorphism, we conclude that $\hT_A\lambda Y\circ \hT_A\ev TY = \htau {TY} \circ \hlambda TY \circ \hpi\hT TY$, so 
\begin{eqnarray*}
\hvarpi Y\circ \homega Y\circ\hT_A\tau Y \circ\htau {TY} 
&=& \hT_A\tau Y\circ\htau {TY}\circ \hlambda T Y\circ \hpi  \hT T Y \circ \hchi \hT TY\\
&=&\hT_A\tau Y\circ\htau {TY}\circ \hlambda T Y\circ \biota^B \hT T Y
=\hT_A\tau Y\circ\htau {TY} \, .
\end{eqnarray*}
Since $\hT_A\tau Y\circ\htau {TY}$ is an epimorphism,  $\hvarpi\circ \homega$ is the identity natural transformation on $\hT_AT_B$, and hence $\homega$ is a natural isomorphism with inverse $\hvarpi$.

In the converse direction, apply  Theorem \ref{th.wide} to obtain a Morita context. Let  $\hvarpi:\EM\Yy\B\to \hT_AT_B$ be the inverse of $\homega$ and put $\hchi = \hvarpi \circ \biota^B$. It follows from naturality that $\hchi$ satisfies the required conditions.
\end{proof}

\begin{remark}
A sufficient condition for the existence of a natural transformation $\hchi$ as in part (1)(iii) of Theorem \ref{th.morita} is the existence of a natural transformation $\hve: \Yy \to \hT T$ such that $\hev\circ\hve = \biota^B $. In the case of a ring-theoretic Morita context (see Remark \ref{rem.Moritaring}) this condition expresses exactly that the Morita map $\hev$ is surjective.
\end{remark}

\begin{remark}
In case the forgetful functors $U_A$ and $U_B$ preserve coequalisers and  $\EM\Xx\A$ and $\EM\Yy\B$ (or equivalently $\Xx$ and $\Yy$) have all coequalisers, Theorem~\ref{th.morita} has the following slightly stronger formulation, which follows directly from the observations made in Remark \ref{rem.morita}.

\noindent {\em
There is a bijective correspondence between the following data:
\begin{enumerate}[(i)]
\item pairs of adjoint functors (resp.\ equivalences of categories) $(T_B:\EM\Yy\B\to\EM\Xx\A,\hT_A:\EM\Xx\A\to\EM\Yy\B)$ such that $\hT_A$ preserves coequalisers;
\item right wide Morita contexts $(T_B,\hT_A,\omega,\homega)$ such that $\homega$ is a natural isomorphism (resp.\ $\homega$ and $\omega$ are natural isomorphisms);
\item Morita contexts $(\A,\B,\T,\bhT,\ev,\hev)$ such that $\hT$ and $T$ preserve coequalisers and there exist a natural transformation $\hchi:\Yy\to \hT_AT$ such that $\hpi \circ\hchi= \biota^B$ (resp.\ there exists as well a natural transformation $\chi:\Xx\to T_B\hT$ such that $\pi\circ\chi=\biota^A$).
\end{enumerate}
}
\end{remark}

\section{Examples and applications}\label{sec.examples}
In this section we apply the criterion for moritability to specific situations of one adjunction, ring-theoretic Morita contexts,  and herds and pre-torsors.

\subsection {Blowing up one adjunction}

Consider an adjunction $(L:\Xx\to \Tt,R:\Tt\to \Xx)$ with unit $\eta$ and counit $\varepsilon$. Let $\C=(LR,L\eta R,\varepsilon)$ be the associated comonad on $\Tt$ and $\A=(RL,R\varepsilon L,\eta)$ the associated monad on $\Xx$; see Section~\ref{sec.eil-moo}. Associated to the Eilenberg-Moore category of $\C$-coalgebras $\EMC \Tt \C$, there is a second adjunction $(U^C:\EMC \Tt \C\to \Tt,\frc C:\Tt\to \EMC \Tt \C)$ with unit $\nu$ and counit $\zeta$, and  hence there is an object $(\Tt,(L,R),(U^C,\frc C))$ in the category $\Adj(\Xx,\EMC \Tt \C)$. The adjunction $(U^C,\frc C)$ induces  a monad $\B=(\frc CU^C,\frc C\zeta U^C, \nu)=(LR, LR\varepsilon, \nu)$ on $\EMC \Tt \C$. Therefore, there is a Morita context
\begin{equation}\label{moritablow}
\TT=(\A,\B,\T, \bhT,\ev=R\varepsilon L, \hev= LR\varepsilon),
\end{equation}
where $T= R$ and $\hT=LRL$, and the corresponding  Eilenberg-Moore category $\EMB \TT$ can be constructed. This leads to  the following diagram of functors.
\[
\xymatrix{
\Xx \ar@<.5ex>[ddrr]^-{L} \ar@<.5ex>[dddd]^-{\hT} \ar@<.5ex>[rrrrrr]^-{F_A} &&&&&& \EM \Xx \A \ar@<.5ex>[llllll]^-{U_A}  \ar@<.5ex>[dddd]^-{\Lambda} \ar@<.5ex>[ddll]^-{H} \\   \\
&& \Tt  \ar@<.5ex>[uull]^-{R}  \ar@<.5ex>[ddll]^-{\frc C} \ar[rr]^-K \ar@/^/[rrrruu]^-{K_A} \ar@/_/[rrrrdd]^-{K_B}
&& \EMB \TT \ar@<.5ex>[uurr]^-{V_A} \ar[ddrr]^{V_B} \\  \\
\EMC \Tt \C \ar@<.5ex>[uurr]^-{U^C} \ar@<.5ex>[uuuu]^-{T} \ar@<.5ex>[rrrrrr]^-{F_B}
&&&&&& (\EMC \Tt \C)\sp \B \ar@<.5ex>[llllll]^-{U_B} 
}
\]
The adjunctions $(F_A,U_A)$ and $(F_B,U_B)$ are the usual free algebra--forgetful adjunctions associated to a monad (see Section~\ref{sec.eil-moo}).  The category $(\EMC \Tt \C)\sp \B$ is the category of dual descent data and consists of triples $(Y,\rho^Y,\tau^Y)$, where $(Y,\rho^Y:Y\to LRY)$ is an object in $\EMC \Tt \C$, and $\tau^Y: LRY\to Y$ is a morphism in $\EMC \Tt \C$ that satisfies the following conditions $\tau^Y\circ \rho^Y= Y$ and $\tau^Y\circ LR\tau^Y= \tau^Y\circ LR\varepsilon Y$.
Finally, the functor $\Lambda:\EM \Xx \A\to (\EMC \Tt \C)\sp \B$ is defined as follows: for all objects $(X,\rho^X)$ in $\EM \Xx \A$, set $\Lambda(X,\rho^X) := (LX, L\eta X, L\rho^X)$.
Here $V_A$ and $V_B$ denote the obvious forgetful functors. The functor $H$ can now be defined as follows:
$$H(X,\rho^X)=((X,\rho^X),\Lambda(X,\rho^X), \rho^X,L\rho^X).$$

\begin{lemma}\label{lem.HV}
$(H,V_A)$ is a pair of adjoint functors, and $H$ is fully faithful. Furthermore, there is a natural transformation $\gamma:\Lambda V_A\to V_B$, for all objects  $\X = ((X,\rho^X), (Y,\rho^Y,\tau^Y), v,w)$ in $\EMB \TT$,  given by 
$\gamma\X= w\circ L\eta X :L X\to Y$.
\end{lemma}
\begin{proof}
We construct  the unit $\alpha$ and the counit $\beta$ of this adjunction. Take an object $(X,\rho^X)$ in $\EM \Xx \A$ and define $\alpha X :(X,\rho^X)\to (X,\rho^X)$ as the identity. For all objects $\X\in\EMB \TT$, define
\begin{eqnarray*}
\beta\X=(f\X,g\X): HV_A(\X)&\to& \X,\\
((X,\rho^X), (LX, L\eta X, L\rho^X) , L\rho^X, \rho^X)&\to& ((X,\rho^X),(Y,\rho^Y,\tau^Y),v,w),\\
f \X=X:X\to X,&\ & g \X=w\circ L\eta X :L X\to Y\, .
\end{eqnarray*}
 Combining naturality and the adjunction properties of the unit $\eta$ and counit $\varepsilon$ with the diagrams $(12)-(13)$, one can check that $f$ and $g$ are indeed morphisms in $\EMB \TT$.

Now define $\gamma$ by putting $\gamma \X = g \X$,  for all $X\in \EMB \TT$.
\end{proof}

\begin{proposition}
With notation as above, $K$ is an equivalence of categories if and only if $K_A$ is an equivalence of categories and $\gamma:\Lambda V_A\to V_B$ is a natural isomorphism.
\end{proposition}
\begin{proof}
Assume that $K$ is an equivalence. Up to an isomorphism, any  $\X\in \EMB \TT$ has the form $\X = KZ$, for some $Z\in \Tt$, i.e.\ 
$$\X= ( (RZ,R\varepsilon Z), (LRZ, L\eta RZ, LR\varepsilon Z), R\varepsilon Z, LR\varepsilon Z).$$
The natural transformation  $\gamma : L V_A \to V_B$ described in Lemma~\ref{lem.HV} comes out as
\[
\xymatrix{
\gamma \X : LRZ \ar[rr]^-{L\eta RZ} && LRLRZ \ar[rr]^-{LR\varepsilon Z} && LRZ
}.
\]
Since $(L,R)$ is an adjoint pair,  $\gamma \X$ is an isomorphism. By the construction of the comparison functor in Section~\ref{sec.comp},  $K_B= V_B K$ and $K_A= V_A K$. Denote the inverse functor of $K$ by $D$. In view of Proposition~\ref{prop.adjoint.comp}, for an object $\X$ in $\EMB \TT$,  $D\X$ is the following colimit in $\Tt$,
\[
\xymatrix{
LRLX \ar@<-.5ex>[dd]_-{\varepsilon LX} \ar@<.5ex>[dd]^-{L\rho^X} \ar@/^/[ddrrrr]^{w} 
&&&& LRY \ar@/^/[lllldd]^{Lv} \ar@<-.5ex>[dd]_-{\varepsilon Y} \ar@<.5ex>[dd]^-{\tau^Y}\\  \\
LX \ar[drr]^-{d^A\X} &&&& Y \ar[dll]_-{\d^B\X}    \\
&& D\X \, .
}
\]
Since $\gamma$ is a natural isomorphism, this colimit reduces to (each) one of the following isomorphic coequalisers
\[
\xymatrix{
LRLX \ar@<-.5ex>[dd]_-{\varepsilon LX} \ar@<.5ex>[dd]^-{L\rho^X} \ar[rr]^-{LR\gamma\X} &&
 LRY  \ar@<-.5ex>[dd]_-{\varepsilon Y} \ar@<.5ex>[dd]^-{\tau^Y}
\\ \\
LX \ar[rr]^-{\gamma \X} \ar[d] && Y \ar[d] \\
D_A X \ar[rr]^-\cong && D_B Y \, . 
}
\]
Therefore, there are  functors $D_A : \EM \Xx \A\to \Tt$ and $D_B: (\EMC \Tt \C)\sp \B\to \Tt$ such that $D\simeq D_A V_A \simeq D_B V_B$ and $D_A \simeq DH$. These yield natural isomorphisms
\begin{eqnarray}
D_A K_A&\simeq &D_A V_A K\simeq D K\simeq  \Tt  \label{DK} , \\
{\EM \Xx \A} &\simeq &  V_A H \simeq  V_A K D H \simeq K_A D_A \, . \label{KD}  
\end{eqnarray}
The fact that $K$ is an equivalence of categories is used in the last isomorphism of \eqref{DK} and in the second isomorphism of \eqref{KD}. The first isomorphism of \eqref{KD} follows by Lemma~\ref{lem.HV}.
These natural isomorphisms are exactly the unit and counit of the adjunction $(D_A,K_A)$, hence $K_A$ is an equivalence of categories.

Conversely, if $\gamma$ is a natural isomorphism, then, in light of its  construction  in the proof of Lemma~\ref{lem.HV}, the counit of the adjunction $(H,V_A)$ is a natural isomorphism, hence $V_A$ is an equivalence of categories.
Since $K = V_A K_A$, we infer that $K$ is an equivalence of categories as well.
\end{proof}

Next, we apply the results of Section~\ref{sec.Moritath} to the pair of monads described at the beginning of this section.  $\TT$ of \eqref{moritablow} is a Morita context between $\Xx$ and $\EMC\Tt\C$. For any $(X,\rho^X)\in\EM\Xx\A$,  the coequaliser \eqref{eq.TAY} is a contractible coequaliser in $\EMC\Tt\C$:
\[
\xymatrix{
LRLRLX \ar@<.5ex>[rr]^-{LR\varepsilon LX} \ar@<-.5ex>[rr]_-{LRL\rho^X} && LRLX \ar[rr]^-{L\rho^X} 
\ar@(ld,rd)[ll]^-{L\eta RLX}
&& LX \ar@(ld,rd)[ll]^-{L\eta X}  .
}
\]
Hence the functor $\hT_A$ exists and, furthermore, $\hT_A=\Lambda$. For any $(Y,\rho^Y,\tau^Y)\in \EM{(\EMC\Tt\C)}\B$, consider the following pair in $\Xx$ (or in $\EM\Xx\A$)
\begin{equation}\label{eq.TBXblow}
\xymatrix{
RLRY\ar@<.5ex>[rr]^-{R\varepsilon Y} \ar@<-.5ex>[rr]_-{R\tau^Y} && RY \, .
}
\end{equation}
Then $R\rho^Y$ is a common right inverse for $R\varepsilon Y$ and $R\tau^Y$. Furthermore, 
the following is a contractible coequaliser in $\EM{(\EMC\Tt\C)}\B$
$$\xymatrix{
LRLRY \ar@<.5ex>[rr]^-{LR\varepsilon Y} \ar@<-.5ex>[rr]_-{LR\tau^Y} && LRY \ar[rr]^-{\tau^Y} \ar@(ld,rd)[ll]^-{L\eta RY} && Y \ar@(ld,rd)[ll]^-{\rho^Y} \, .
}$$
Hence, \eqref{eq.TBXblow}  is a reflexive $L$-contractible coequaliser pair. Thus if the coequaliser of the pair \eqref{eq.TBXblow} exists (in $\Xx$), it is exactly the coequaliser \eqref{eq.TBX}. 

\begin{proposition}\label{pr.blow}
Consider the following statements.
\begin{enumerate}[(i)]
\item $K_B$ is an equivalence of categories;
\item $\Xx$ contains coequalisers of pairs \eqref{eq.TBXblow} (i.e.\ reflexive $L$-contractible coequaliser pairs) and $AA$ preserves them;
\item The Morita context $\TT$ of \eqref{moritablow} induces an equivalence of categories  $\Xx^\A$ and $\EM{(\EMC\Tt\C)}\B$.
\end{enumerate}
Then $(i)$ implies $(ii)$ implies $(iii)$.
\end{proposition}

\begin{proof}
$\ul{(i)\Rightarrow(ii)}$. Since $K_B$ is an equivalence,  up to an isomorphism,  objects in $\EM{(\EMC\Tt\C)}\B$ are of the form $K_B(Z)=(LRZ,L\eta RZ, LR\varepsilon Z)$ for some $Z\in\Tt$. Consequently the pair \eqref{eq.TBXblow} results in the following contractible coequaliser in $\EMC\Tt\C$
\[
\xymatrix{
RLRLRZ \ar@<.5ex>[rr]^-{R\varepsilon LRZ} \ar@<-.5ex>[rr]_-{RLR\varepsilon Z} && RLRZ \ar[rr]^-{R\varepsilon Z} 
\ar@(ld,rd)[ll]^-{\eta RLRZ}
&& RZ \ar@(ld,rd)[ll]^-{\eta RZ} \, .
}
\]
$\ul{(ii)\Rightarrow(iii)}$. This follows immediately from Theorem~\ref{th.morita} and the above observations.
\end{proof}

\begin{corollary}\label{co.blow}
$K_B$ is an equivalence if and only if $K_A$ is an equivalence and conditon $(ii)$ of Proposition~\ref{pr.blow} holds.
\end{corollary}

\begin{proof}
This follows from Proposition \ref{pr.blow} and the equality $K_B=\Lambda K_A (=\hT_A K_A)$. 
\end{proof}

\begin{remark}
Proposition~\ref{pr.blow} and Corollary~\ref{co.blow} provide an alternative proof of (the dual version of) \cite[Theorem 2.19, Theorem 2.20]{Mes:effdes}. 
\end{remark}

\subsection {Morita theory for rings} \label{Moritaring}
In ring and module theory, a {\em Morita context} is a sextuple $(A,B,M,N,\sigma,\hsigma)$, where $A$ and $B$ are rings, $M$ is an $A$-$B$ bimodule, $N$ is a $B$-$A$ bimodule, $\sigma : M\ot_B N\to A$ is an $A$-$A$ bimodule map, and $\hsigma: N\ot_A M \to B$ is a $B$-$B$ bimodule map rendering commutative the following diagrams
\begin{equation}\label{morita.mod}
\xymatrix{
M\ot_B N\ot_AM \ar[rr]^-{\sigma\ot_AM}\ar[d]_{M\ot_B \hsigma} && A\ot_A M \ar[d] \\
M\ot_B B \ar[rr] && M,} \qquad \xymatrix{
N\ot_A M\ot_BN \ar[rr]^-{\hsigma\ot_BN}\ar[d]_{N\ot_A \sigma} && B\ot_BN \ar[d] \\
N\ot_AA \ar[rr] && N.}
\end{equation}
The unmarked arrows are canonical isomorphisms (induced by actions). With every Morita context one associates a matrix-type Morita ring 
$$
Q = \begin{pmatrix}A & M\\ N& B \end{pmatrix} := \{ \begin{pmatrix}a & m\\ n& b \end{pmatrix} \; | \; a\in A, b\in B, m\in M, n\in N\}, 
$$
with the product given by 
$$
\begin{pmatrix}a & m\\ n& b \end{pmatrix} \begin{pmatrix}a' & m'\\ n'& b' \end{pmatrix} = \begin{pmatrix}aa' + \sigma(m\ot_A n') & am' + mb'\\ na'+ bn'& bb' + \hsigma(n\ot_B m') \end{pmatrix} .
$$
Since $Q$ has two orthogonal idempotents summing up to the identity, its left modules split into direct sums. More precisely, left $Q$-modules correspond to quadruples $(X,Y, \bv, \bw)$,  where $X$ is a left $A$-module, $Y$ is a left $B$-module, $\bv: M\ot_B Y\to X$ is a left $A$-module map and $\bw: N\ot_A X\to Y$ is a left $B$-module map such that the following diagrams
\begin{equation}\label{morita.ring}
\xymatrix{
M\ot_B N\ot_AX \ar[rr]^-{\sigma\ot_AX}\ar[d]_{M\ot_B \bw} && A\ot_A X \ar[d] \\
M\ot_B Y \ar[rr]^\bv && X,} \qquad \xymatrix{
N\ot_A M\ot_BY \ar[rr]^-{\hsigma\ot_BY}\ar[d]_{N\ot_A \bv} && B\ot_BY \ar[d] \\
N\ot_AX \ar[rr]^\bw && Y,}
\end{equation}
commute. The left action of $Q$ on $X\oplus Y$ is given by
$$
\begin{pmatrix}a & m\\ n& b \end{pmatrix} \begin{pmatrix} x \\ y \end{pmatrix} = \begin{pmatrix}ax + \bv(m\ot_B y) \\  by + \bw(n\ot_A x) \end{pmatrix} .
$$
 
Take abelian groups ($\ZZ$-modules) $A$ and $B$. The associated tensor functors $\tA = A\ot - :\Ab \to \Ab$, $\tB= B\ot - : \Ab \to \Ab$ are monads on the category of abelian groups  if and only if $A$ and $B$ are rings. Furthermore $\EM\Ab \tA = {}_A\Mm$ (the category of left $A$-modules) and $\EM\Ab  \tB = {}_B\Mm$. Take abelian groups $M$ and $N$. The tensor functor $T = M\ot - :\Ab \to \Ab$ is an $A\ot(-)$-$B\ot(-)$ bialgebra if and only if $M$ is an $A$-$B$ bimodule. The left action $\lambda: \tA T= A\ot M\ot -\to M\ot -$ is, for all abelian groups $X$, $\lambda X : \lambda\ZZ \ot X$, where $\lambda\ZZ  : A\ot M\to M$ is a left action of $A$ on $M$. Similarly, the right action of $\tB$ on $T$ corresponds to a right action of $B$ on $M$. Symmetrically,  $\hT = N\ot -$ is a $\tB$-$\tA$ bialgebra if and only if $N$ is a $B$-$A$ bimodule. Any $\tA$-$\tA$ bialgebra map $\ev: T\hT = M\ot N\ot - \to A\ot - = \tA$ is fully determined by its value at $\ZZ$ by $\ev X = \ev \ZZ\ot X$; the map $\ev \ZZ: M\ot N\to A$ is an $A$-$A$ bimodule map. If the map $\ev$ is required to satisfy the second of the diagrams in \eqref{balanced}, then the universality of tensor products yields a unique $A$-bimodule map $\sigma: M\ot_BN\to A$ such that
$$
\ev \ZZ :  \xymatrix{  M\ot N \ar[r] & M\ot_B N \ar[r]^-\sigma & A}.
$$
Conversely, any $A$-$A$ bilinear map $\sigma: M\ot_BN\to A $ determines a $\tB$-balanced $\tA$-bialgebra natural map $\ev: T\hT = M\ot N\ot - \to A\ot - = \tA$.
In a symmetric way there is a bijective correspondence between $\tA$-balanced $\tB$-bialgebra maps $\hev: \hT T\to \tB$ and $B$-$B$ bilinear maps $\hsigma: N\ot_A M\to B$. The natural transformations $\ev$, $\hev$ satisfy conditions \eqref{morita} if and only if the corresponding maps $\sigma$, $\hsigma$ satisfy conditions \eqref{morita.mod}. These observations establish bijective correspondence (in fact, an isomorphism of categories) between module theoretic Morita contexts $(A,B,M,N,\sigma,\hsigma)$ and objects in $\Mor (A\ot -, B\ot -)$. 

Once an object $\TT$ in $\Mor (A\ot -, B\ot -)$ is identified with a module-theoretic Morita context $(A,B,M,N,\sigma,\hsigma)$ one can compute the corresponding Eilenberg-Moore category $\EMAb \TT$. An object in $\EMAb \TT$ consists of a left $A$-module $X$ (an algebra of the monad $A\ot -$) and a left $B$-module $Y$ (an algebra of the monad $B\ot -$), and two module maps $v: M\ot Y\to X$ and $w : N\ot X\to Y$. The commutativity of diagrams \eqref{vwass2} yields unique module maps $\bv : M\ot_B Y\to X$ and $\bw : N\ot_A X\to Y$ such that the following diagrams
$$
\xymatrix{ M\ot Y \ar[rr]^-v\ar[dr] && X \\ & M\ot_B Y\ar[ru]_\bv , & } \qquad \xymatrix{ N\ot X \ar[rr]^-w\ar[dr] && Y \\ & N\ot_A X\ar[ru]_\bw  & }
$$
commute. Diagrams \eqref{vwass1}  for $v$ and $w$ are equivalent to diagrams \eqref{morita.ring} for $\bv$ and $\bw$. This establishes an identification (isomorphism) of the Eilenberg-Moore category $\EMAb \TT$ with the category of left modules of the corresponding matrix Morita ring $Q$. 

With the interpretation of $\EMAb \TT$ as left $Q$-modules  ${}_Q\Mm$, the functors $U_\tA$, $U_\tB$ constructed in Section~\ref{mor.to.adj} are forgetful functors ${}_Q\Mm \to \Ab$, $U_\tA(X\oplus Y) = X$, $U_\tB(X\oplus Y) = Y$. The functors $G_\tA$, $G_\tB$ send an abelian group $X$ to its tensor product with respective columns in $Q$. More precisely, $G_\tA(X) = A\ot X \oplus N\ot X$ with the multiplication by $Q$,
$$
\begin{pmatrix}a & m\\ n& b \end{pmatrix} \begin{pmatrix}a' \ot x \\ n' \ot y \end{pmatrix} = \begin{pmatrix}aa'\ot x + \sigma(m\ot_A n')\ot y \\ na'\ot x+ bn'\ot y \end{pmatrix} ,
$$
 while $G_\tB(X) = M\ot X\oplus B\ot X$ with the action of $Q$,
 $$
\begin{pmatrix}a & m\\ n& b \end{pmatrix} \begin{pmatrix} m'\ot x \\ b'\ot y \end{pmatrix} = \begin{pmatrix} am'\ot x + mb'\ot y\\  \hsigma(n\ot_B m')\ot x + bb'\ot y \end{pmatrix} ,
$$
for all $m,m'\in M$, $n,n'\in N$, $a,a'\in A$, $b,b'\in B$ and  $x,y\in X$.

The construction presented in this section can be repeated with $\Xx$ chosen to be the category ${}_{k_A}\Mm$ of left modules over a ring $k_A$ and $\Yy$ the category of left modules over a ring $k_B$. All the functors $\tA$, $\tB$, $T$, $\hT$ can be chosen as tensor functors with the tensor product over respective rings $k_A$ or $k_B$. For example, take a $k_A$-$k_A$ bimodule $A$ and define $\tA$  as a functor $-\ot_{k_A}A:{}_{k_A}\Mm\to {}_{k_A}\Mm$. $\tA$ is a monad if and only if $A$ is a $k_A$-ring. Similarly choose $\tB = - \ot_{k_B}B$ for a $k_B$-ring $B$. Since modules over a $k_A$-ring $A$ coincide with modules of the ring $A$, one can choose further an $A$-$B$ bimodule $M$ and $B$-$A$ bimodule $N$ and proceed as above, taking care to  decorate suitably tensor products with $k_A$ and $k_B$.

\subsection{Categories with binary coproducts}
The characterisation of the Eilenberg-Moore category of  a module-theoretic Morita context described in Section~\ref{Moritaring} can be seen as a special case of the following situation. Assume that categories $\Xx$ and $\Yy$ have binary coproducts. Take a Morita context $\TT=(\A,\B,\T,\bhT,\ev,\hev)$ on $\Xx$ and $\Yy$ which functors $A$, $B$, $T$ and $\hT$ preserve (binary) coproducts. These data lead to the following monad $\Q = (Q, \m ,\biota)$ on the product category $\Xx\times \Yy$. The functor $Q$ is defined as $Q : (X,Y)\mapsto (AX+ T Y,BY+ \hT X)$. The unit $\biota$ of $\Q$ is the composite:
$$
\xymatrix{\biota(X,Y) : (X,Y) \ar[rr]^-{(\biota^AX,\biota^BY)} && (AX,BY) \ar[rr] && (AX+ T Y,BY+ \hT X),}
$$
where the second morphism consists of the canonical injections $AX \to AX + T Y$ and $BY \to BY+ \hT X$. The multiplication, $\m(X,Y)$, 
$$
(AAX + AT Y + T B Y + T \hT X, BBY + B\hT X + \hT AX+\hT T Y) \to (AX+ T Y,BY+ \hT X),
$$
combines multiplications in $A$ (and $B$) with actions of monads on $T$ (and $\hT$) and with $\ev$ (and $\hev$). That is $AAX+ T \hT X + AT Y + T B Y  \to AX+ T Y$ is given as a sum $(\m^A+ \ev)X +(\lambda  + \rho )Y$, where $(\m^A+ \ev)X: AAX+ T \hT X\to AX$ and $(\lambda  + \rho )Y: AT Y + T B Y  \to T Y$  are the unique fillers in the diagrams
$$
\xymatrix{ 
& AAX +T\hT X &\\
  AAX\ar[ur] \ar[dr]_{\m^A X} & & T\hT X  , \ar[ul] \ar[dl]^{\ev X} \\
  & AX & } \qquad \xymatrix{ 
& ATY +TB Y &\\
  ATY \ar[ur] \ar[dr]_{\lambda Y} & & TB Y  \ar[ul] \ar[dl]^{\rho Y} \\
  & TY . & } 
  $$
 The second component of $\m(X,Y)$ is defined similarly. In this case, the Eilenberg-Moore category for the Morita context $\TT$ is isomorphic to the Eilenberg-Moore category of algebras of the monad $\Q$, i.e.\ $(\Xx, \Yy)^\TT \cong (\Xx\times\Yy) ^\Q$.

 Given a double  adjunction  $\tt=(\Tt,(L_A,R_A),(L_B,R_B)) \in \Adj(\Xx,\Yy)$, in which also $\Tt$ has binary coproducts which are preserved by $R_A,R_B$, the functor $\langle R_A , R_B\rangle : \Tt \to \Xx\times \Yy$, $Z\mapsto (R_AZ, R_BZ)$ has a left adjoint $[L_A , L_B] :(X,Y)\mapsto L_AX + L_BY$. The monad defined by the adjunction $[L_A , L_B] \dashv \langle R_A , R_B\rangle$ is simply the monad $\Q$ corresponding to the Morita context $\Upsilon(\tt)$ (see Section~\ref{adj.to.mor}). In view of the above identification of Eilenberg-Moore categories, the moritability of $(R_A,R_B)$ is equivalent to the monadicity of $\langle R_A , R_B\rangle$ and thus is decided by the classical Beck theorem.

\subsection {Formal duals}
As recalled in Section~\ref{sec.morita}, given a monad $\A= (A,\m^A,\biota^A)$ on $\Xx$ and a monad $\B= (B,\m^B,\biota^B)$ on $\Yy$ a pair of formally dual bialgebras is a pair of bialgebra functors $T: \Yy\to \Xx$, $\hT: \Xx\to \Yy$ equipped with natural bialgebra transformations $\ev: T\hT\to A$, $\hev: \hT T\to B$ that satisfy compatibility conditions expressed by diagrams \eqref{morita}. In other words, a pair of formally dual bialgebra functors is the same as an {\em unbalanced Morita context}. Morphisms between pairs of formally dual bialgebras are quadruples consisting of two monad morphism and two bialgebra morphisms which satisfy the same conditions as morphisms between Morita contexts. This defines a category $\Dual(\Xx,\Yy)$ of which $\Mor(\Xx,\Yy)$ is a full subcategory. Thus the functor $\Upsilon : \Adj (\Xx,\Yy) \to \Mor(\Xx,\Yy)$ described in Section~\ref{adj.to.mor} extends to the functor $\Adj (\Xx,\Yy) \to \Dual(\Xx,\Yy)$.

\subsection {Herds versus pretorsors}
Following \cite[Appendix]{BrzVer:herd}, a {\em herd functor} is a pair of formally dual bialgebra functors $\TT=(\A,\B,\T,\bhT,\ev,\hev)$ (i.e.\ an object of $\Dual(\Xx,\Yy)$) together with a natural transformation $\gamma: T\to T\hT T$ 
rendering commutative the following diagrams
\begin{equation}\label{herd}
\xymatrix{ & T  \ar[dd]^\gamma \ar[dl]_{T\biota^B}\ar[dr]^{\biota^A T}& \\
TB & & AT  \ ,  \\
& T \hT T  \ar[ul]^{T\hev}\ar[ur]_{\ev T} &} \qquad
\xymatrix{T \ar[rr]^-{\gamma}\ar[dd]_{\gamma} && T\hT T\ar[dd]^{T\hT \gamma} \\ && \\
T\hT T \ar[rr]^-{\gamma \hT T} && {T\hT T \hT T}\ .} 
\end{equation}
The map $\gamma$ is called a {\em shepherd}. A morphism between herds $\phi: (\A,\B,\T,\bhT,\ev,\hev, \gamma)\to (\A',\B',\T',\bhT',\ev',\hev', \gamma')$  is a morphism  $\phi =(\phi^1,\phi^2,\phi^3,\phi^4)$ in $\Dual(\Xx,\Yy)$ compatible with shepherds $\gamma$ and $\gamma'$, i.e.\ whose third and fourth components make the following diagram
\begin{equation}\label{herd.mor}
\xymatrix{
T \ar[rr]^-{\gamma} \ar[d]_{\phi^3} && T\hT T \ar[d]^{\phi^3\phi^4\phi^3} \\
T' \ar[rr]_-{\gamma'} && T'\hT'T' \ 
}
\end{equation}
commute, where $\phi^3\phi^4\phi^3 = \phi^3\hT' T'\circ T \phi^4 T'\circ T\hT\phi^3$ denotes the Godement product. The category of herd functors between categories $\Xx$ and $\Yy$ is denoted by $\Herd (\Xx, \Yy)$. This category contains a full subcategory $\bHerd (\Xx, \Yy)$ of {\em balanced herds} with objects those herds, whose underlying formally dual pair is a Morita cotext (i.e.\ characterised by the forgetful functor $\bHerd (\Xx, \Yy)\to \Mor (\Xx, \Yy)$).

Following \cite[Definition~4.1]{BohMen:pretor}, a {\em pre-torsor} is an object $\tt=(\Tt,(L_A,R_A),(L_B,R_B))$ of $\Adj(\Xx,\Yy)$   together with a natural transformation $\tau: R_AL_B\to R_AL_B R_BL_A R_AL_B$ 
rendering commutative the following diagrams
\begin{equation}\label{pretor1}
\xymatrix{ & R_AL_B  \ar[dd]^\tau \ar[dl]_{R_AL_B\eta^B}\ar[dr]^{\eta^AR_AL_B}& \\
R_AL_B R_BL_B & & R_AL_AR_AL_B  \ ,  \\
& R_AL_B  R_BL_A R_AL_B  \ar[ul]^{R_AL_BR_B\varepsilon^AL_B}\ar[ur]_{R_A\varepsilon^BL_A R_AL_B} &} 
\end{equation}
\begin{equation}\label{pretor2}
\xymatrix{R_AL_B \ar[rrr]^-{\tau}\ar[d]_{\tau} &&& R_AL_BR_BL_A R_AL_B\ar[d]^{R_AL_BR_BL_A \tau} \\
R_AL_BR_BL_A R_AL_B \ar[rrr]^-{\tau R_BL_AR_AL_B} &&& {R_AL_BR_BL_A R_AL_B R_BL_A R_AL_B}\ .} 
\end{equation}
A morphism from a pre-torsor $(\Tt,(L_A,R_A),(L_B,R_B), \tau)$ to   $(\Tt',(L'_A,R'_A),(L'_B,R'_B), \tau')$ is a morphism $F$ in  $\Adj(\Xx,\Yy)$  that is compatible with the structure maps $\tau$ and $\tau'$. The compatibility condition is expressed as the equality
\begin{equation}\label{pretor.mor}
(R_AbR_BaR_Ab)\circ \tau = \tau'\circ R_A b,
\end{equation}
where $a$ and $b$ are defined by \eqref{eq:a,b} and the shorthand notation for the Godement product is used.  The category of pre-torsors on  categories $\Xx$ and $\Yy$ is denoted by $\pretor (\Xx, \Yy)$. 

Even the most perfunctory comparison of diagrams \eqref{herd} with \eqref{pretor1}-\eqref{pretor2} reveals that the functor $\Upsilon$ applied to the double adjunction $\tt$ underlying a pre-torsor $(\tt, \tau)$ yields a (balanced) herd with a shepherd $\tau$, i.e.\ $(\Upsilon\tt, \tau)$ is a herd. The definition of $\Upsilon$ on morphisms  immediately affirms that the condition \eqref{pretor.mor} for $F$ implies condition \eqref{herd.mor} for $\Upsilon F$. Thus $\Upsilon$ yields the functor
$$
\bUpsilon : \pretor(\Xx, \Yy) \to \bHerd(\Xx, \Yy) \subseteq \Herd(\Xx, \Yy), \qquad (\tt,\tau)\mapsto (\Upsilon\tt, \tau).
$$

In the converse direction, the functor $\Gamma$ constructed in Section~\ref{mor.to.adj} yields the functor
$$
\bGamma : \bHerd(\Xx, \Yy)  \to  \pretor(\Xx, \Yy), \qquad (\TT, \gamma) \mapsto (\Gamma\TT, \gamma).
$$
The key observation is that $T = U_AG_B$ and $\hT = U_BG_A$, hence the shepherd $\gamma$ of the herd $T$ becomes the natural transformation $\tau$ for the corresponding pre-torsor $(\EMM\TT, (G_A,U_A),(G_B,U_B))$; see the definition of $U_A$, $G_A$, $U_B$, $G_B$ in Section~\ref{mor.to.adj}. 

As (implicitly) calculated in the proof of Proposition~\ref{prop.adjoint}, the natural trasformations $a$ and $b$ corresponding to the comparison functor $K: \Tt\to \EMM \TT$ are identity maps, hence the condition \eqref{pretor.mor} is trivially satisfied, so, for any pre-torsor, $K$ is a morphism of pre-torsors. Thus Proposition~\ref{prop.adjoint} and Corollary~\ref{cor.equiv} immediately imply

\begin{corollary}\label{cor.pretor}
$(\bGamma, \bUpsilon)$ is an adjoint pair and $\bGamma$ is a full and faithful functor. Furthermore, $(\bGamma, \bUpsilon)$ is a pair of inverse equivalences between categories of pre-torsors and balanced herd functors if and only if, for all pre-torsors $(\Tt,(L_A,R_A),(L_B,R_B), \tau )\in \pretor(\Xx,\Yy)$, $(R_A,R_B)$ is a moritable pair. In particular, if $(\Gamma, \Upsilon)$ is a pair of inverse equivalences, then so is $(\bGamma, \bUpsilon)$.
\end{corollary}

\section{Remarks on dualisations and generalisations}

\subsection {Dualisations}
There are various ways in which the categories studied in preceding sections and thus results described there can be (semi-)dualised.

One can define the category $\Adj^o(\Xx,\Yy)$, whose objects are pentuples (or triples)
$(\Tt,(L_A,R_A),(L_B,R_B))$, where $\Tt$ is a category and
$(L_A:\Xx\to\Tt,R_A:\Tt\to\Xx)$ and $(L_B:\Tt\to\Yy,R_B:\Yy\to\Tt)$ are adjunctions.
Morphisms are defined in the natural way.

The category $\Adj^c(\Xx,\Yy)$ is defined by taking objects $(\Tt,(L_A,R_A),(L_B,R_B))$, where $\Tt$ is a category and
$(L_A:\Tt\to\Xx,R_A:\Xx\to\Tt)$ and $(L_B:\Yy\to\Tt,R_B:\Tt\to\Yy)$ are adjunctions.
 
Finally, the category $\Adj^{o,c}(\Xx,\Yy)$ has objects $(\Tt,(L_A,R_A),(L_B,R_B))$, where $\Tt$ is a category and $(L_A:\Tt\to\Xx,R_A:\Xx\to\Tt)$ and $(L_B:\Tt\to\Yy,R_B:\Yy\to\Tt)$ are adjunctions.

On the other hand, one can consider the category of Morita-Takeuchi contexts, whose objects are sextuples $(\C,\D,\P,\bhP,\cov,\ol{\cov})$ consisting of two comonads, two bicoalgebra functors, and two bicolinear cobalanced natural transformations satisfying compatibility conditions dual to those in Section~\ref{sec.morita}. Also, one can consider an intermediate version, where the first two objects in the sextuple are a  monad and a comonad respectively.

By (semi-)dualising the results of the previous sections,  functors between the respective categories with pairs of adjunctions and the respective categories of contexts can be constructed. Appropriate Eilenberg-Moore categories for  various contexts can be defined thus yielding  a converse construction and leading to the definition of a comparison functor in each case, and to the solution of the corresponding moritability problem.

\subsection {Bicategories}
Adjoint pairs, Morita  and Takeuchi contexts have a natural formulation within the framework of bicategories; see, for example, the bicategorical formulation of wide Morita contexts in \cite{ElK:wid}. We believe that our work can, taking into account the needed (computational) care but without any conceptual problems, be transferred to this (more general) setting. However, we preferred to formulate the results of this paper in the present way, as this presentation might be clearer, more accessible and we believe that even in this generality it covers already enough interestering examples and applications.

\section*{Acknowledgements}
The authors would like to thank Viola Bruni and George Janelidze for helpful comments. The research of A.\ Vazquez Marquez is supported by the CONACYT  grant no.\  208351/303031.

\end{document}